\documentclass{siamart220329}

\usepackage{lipsum}
\usepackage{amsfonts}
\usepackage{graphicx}
\usepackage{epstopdf}
\usepackage{algorithmic}
\usepackage{bm}
\usepackage{threeparttablex}
\usepackage{booktabs}
\usepackage{babel}
\usepackage{multirow}
\usepackage{threeparttable}
\usepackage{comment}

\ifpdf
  \DeclareGraphicsExtensions{.eps,.pdf,.png,.jpg}
\else
  \DeclareGraphicsExtensions{.eps}
\fi

\newsiamremark{remark}{Remark}
\newsiamremark{hypothesis}{Hypothesis}
\newsiamthm{claim}{Claim}

\headers{Density estimation via PSKK method}{Ziyang Ye, Haoyuan Tan, Xiaoqun Wang, Zhijian He}

\title{Density estimation via periodic scaled Korobov kernel method with exponential decay condition\thanks{Submitted to the editors DATE.
\funding{The work of the first and third authors was funded by the National Science Foundation of China (grant 72571153). The work of the fourth authors was funded by the Guangdong Basic and Applied Basic Research
Foundation (grant
2025A1515011888 and 2024A1515011876).}}}

\author{Ziyang Ye\thanks{Department of Mathematical Sciences, Tsinghua University, Beijing 100084, People’s Republic
of China 
(\email{yzy23@mails.tsinghua.edu.cn},\email{wangxiaoqun@mail.tsinghua.edu.cn}).}
\and Haoyuan Tan\thanks{School of Mathematics, South China University of Technology,
Guangzhou 510641, People's Republic of China
  (\email{202410188273@mail.scut.edu.cn}).}
  \and Xiaoqun Wang\footnotemark[2]
\and Zhijian He\thanks{Corresponding author. School of Mathematics, South China University of Technology,
Guangzhou 510641, People's Republic of China (\email{hezhijian@scut.edu.cn}).}
}

\usepackage{amsopn}

\ifpdf
\hypersetup{
  pdftitle={Density estimation via periodic scaled Korobov kernel method with exponential decay condition},
  pdfauthor={Ziyang Ye, Haoyuan Tan,Xiaoqun Wang, Zhijian He}
}
\fi

\begin{document}
\maketitle

\begin{abstract}
We propose the periodic scaled Korobov kernel (PSKK) method for nonparametric density estimation on 
$\mathbb{R}^d$. By first wrapping the target density into a periodic version through modulo operation and subsequently applying kernel ridge regression in scaled Korobov spaces, we extend the kernel approach proposed by Kazashi and Nobile (SIAM J. Numer. Anal., 2023) and eliminate its requirement for inherent periodicity of the density function. This key modification enables effective estimation of densities defined on unbounded domains. We establish rigorous mean integrated squared error (MISE) bounds, proving that for densities with smoothness of order $\alpha$ and exponential decay, our PSKK method achieves an $\mathcal{O}(M^{-1/(1+1/(2\alpha)+\epsilon)})$ MISE convergence rate with an arbitrarily small $\epsilon>0$. While matching the convergence rate of the previous kernel approach, our method applies to non-periodic distributions at the cost of stronger differentiability and exponential decay assumptions. Numerical experiments confirm the theoretical results and demonstrate a significant improvement over traditional kernel density estimation in large-sample regimes.

\end{abstract}

\begin{keywords}
density estimation, periodization, kernel method, scaled Korobov space
\end{keywords}

\begin{MSCcodes}
62G07, 65D40
\end{MSCcodes}

\section{Introduction}
Density estimation is a fundamental challenge in statistics with broad applications in uncertainty quantification \cite{ditkowski2020density,mcdonald2021review,nasri2025novel,rajendran2020uncertainty}, machine learning \cite{amrani2021noise,steininger2021density,wang2019nonparametric,zandieh2023kdeformer}, and data analysis \cite{alquicira2021nebulosa,d2021automatic,hallin2022classifying,silverman1986density}. Unlike parametric methods \cite{varanasi1989parametric}, which assume a specific form for the distribution (e.g., Gaussian), the density estimation in which we are interested are often non-parametric or semi-parametric, allowing them to adapt to the complexity of the data without strong assumptions about the distribution. Classical methods, such as kernel density estimation (KDE), are effective in low dimensions but may suffer from the curse of dimensionality \cite{chen2017tutorial}. Generally speaking, it is difficult to solve density estimation problems exceeding 6 dimensions using KDE methods. Additionally, traditional methods often perform poorly when the density function is multi-modal \cite{silverman1986density}.

With these challenges in mind, in this paper we consider the density estimation problem over $\mathbb{R}^d$. Let $(\Omega,\mathcal{F},\mathbb{P})$ be a probability space and $(\mathbb{R}^d,\mathcal{B}(\mathbb{R}^d))$ be a measurable space, where $\mathcal{B}(\mathbb{R}^d)$ denotes the Borel sets in $\mathbb{R}^d$. Given $M$ independent random vectors $Y_1,\ldots,Y_M:\Omega\to \mathbb{R}^d$ that follow an identical distribution defined by a density $f$ with respect to the Lebesgue measure on $\mathbb{R}^d$, we aim to obtain an estimator $\bar{f}$ so that the mean integrated squared error (MISE), defined as 
\begin{equation}
    \mathbb{E}\left[\int_{\mathbb{R}^d}|\bar{f}(\bm{x})-f(\bm{x})|^2d\bm{x}\right],\nonumber
\end{equation}
is small. For this error, traditional methods like KDE with a second-order kernel achieves an convergence rate of $\mathcal{O}(M^{-\frac{4}{d+4}})$ for the asymptotic MISE (see, for example, \cite[page 100]{wand1994kernel}), which degrades exponentially as the dimension $d$ increases. Beyond this traditional approach, numerous alternative density estimation methods have emerged, such as the maximum a posteriori (MAP) estimator \cite{griebel2010finite,wong2012maximum}, techniques based on minimizing Fisher divergence \cite{sriperumbudur2017density}, supervised deep learning approaches \cite{bos2024supervised}, and deep generative neural networks \cite{liu2021density}. While these methods offer advantages in overcoming scalability or adaptability limitations associated with KDE, a quantitative analysis of their MISE convergence rates in general dimensions remains largely lacking.

In order to obtain a higher MISE convergence rate, a regularized least squares approach is considered in many works \cite{hegland2000finite,kazashi2023density,peherstorfer2014density,roberts2008note}. This approach supposes that the density function $f$ lies in some inner product space $W$ and want to compute finite-sample approximations of $f$ by discretizing the minimization problem
\begin{equation}\label{min_pro}
    \min_{g\in V}\Vert f-g\Vert^2_{L^2(\mathbb{R}^d)}+\lambda\Vert g\Vert^2_W,
\end{equation}
where $V$ is a finite dimensional subspace of $W$, $\lambda>0$ is a regularization parameter, and $\Vert\cdot\Vert_W$ denotes the norm induced by the inner product $\left\langle\cdot,\cdot\right\rangle_{W}$ on $W$. 
The discretized minimization problem~(\ref{min_pro}) is equivalent to find $g\in V$ such that
\begin{equation}\label{approx_tpye}
    \left\langle g,v\right\rangle_{L^2(\mathbb{R}^d)}+\lambda \left\langle g,v\right\rangle_{W}=\frac{1}{M}\sum_{m=1}^{M}v(Y_m),\quad\text{for all }v\in V.
\end{equation}
The choice of the finite-dimensional subspace $V$ fundamentally determines the form of the approximation and influences its convergence behavior. Different approaches have employed diverse subspaces, including finite element subspaces \cite{hegland2000finite}, sparse grid function spaces \cite{roberts2008note}, and linear span of hat functions centered at adaptive sparse grids \cite{peherstorfer2014density}. It is worth noting that the density estimation in these studies (\cite{hegland2000finite, peherstorfer2014density, roberts2008note}) is confined to bounded domains rather than the entire $\mathbb{R}^d$. Among them, only \cite{roberts2008note} derives an MISE convergence rate of $\mathcal{O}(|\log M|^{2d}M^{-\frac{4}{5}})$, which does not attain the optimal lower bound $\mathcal{O}(M^{-1})$ given in \cite{boyd1978lower}.

More recently, Kazashi and Nobile \cite{kazashi2023density} applied the approximation~(\ref{approx_tpye}) to density functions in the reproducing kernel Hilbert space (RKHS). This approach is widely used in density ratio estimation \cite{kanamori2009least,kanamori2012statistical,sugiyama2012density,zellinger2023adaptive} but remains relatively rare for obtaining estimators and analyzing MISE in density estimation problems. In Kazashi and Nobile's setting the function space $W$ is an RKHS with a reproducing kernel $K:D\times D\to \mathbb{R}$, where $D$ is a general set, and the finite-dimensional subspace $V$ is chosen to be 
\begin{equation}\label{V_N}
    V=V(X)= \text{span}\{K(\bm{x}_k,\cdot)|k=1,\ldots,N\},
\end{equation}
where $X=\{\bm{x}_1,\ldots,\bm{x}_N\}\subset D$ is a preselected point set. They developed a theoretical framework for analyzing MISE for a general reproducing kernel $K$ that admits an $L^2$-orthonormal expansion. In particular, for the density function $f$ in the weighted Korobov space over $[0,1]^d$, they proved that the total MISE is $\mathcal{O}(M^{-1/(1+1/(2\alpha)+\epsilon)})$, where $\alpha$ is the smoothness parameter of the Korobov space and $\epsilon>0$ is an arbitrary positive number. However, the assumption that $f$ lies in the weighted Korobov space forces the density function $f$ to be a periodic function on $[0,1]^d$, which limits the applicability of the error bound.

To overcome the limitations associated with periodicity and compact support of the density functions discussed in \cite{kazashi2023density}, we propose the periodic scaled Korobov kernel (PSKK) method. Specifically, within the density estimation framework of \cite{kazashi2023density}, we employ the kernel $K=K_{\alpha,a,d} :[-a,a]^d\to\mathbb{R}$, which is the kernel of the scaled Korobov space over $[-a,a]^d$ with smoothness parameter $\alpha$, as defined by Nuyens and Suzuki \cite{nuyens2023scaled}. Similar to \cite{kazashi2023density}, we construct the pre-selected point set $X$ using the component-by-component (CBC) method from \cite{cools2020lattice}, with the additional step of linearly scaling the resulting lattice point set into $[-a,a]^d$. We then establish the MISE bound with the scaling parameter $a$. In order to handle the periodicity that Korobov spaces need, we artificially map the random vector to $[-a,a)^d$ via modulo $2a$ operation. More specifically, for each $1\le m \le M$, we transform the random vector $Y_m=(Y_{m,1},\ldots,Y_{m,d})^\top$ into $\widetilde{Y}_m=(\widetilde{Y}_{m,1},\ldots,\widetilde{Y}_{m,d})^\top$ by defining
\begin{equation}
    \widetilde{Y}_{m,j}= (Y_{m,j}\ \text{mod $2a$)}-a\in [-a,a),\nonumber
\end{equation}
for all $1\le j\le d$. Then $\widetilde{Y}_1,\ldots,\widetilde{Y}_M$ are iid random vectors with periodic density function
\begin{equation}\label{wrapped_f}
    \widetilde{f}(\bm{x})=\sum_{\bm{k}\in\mathbb{Z}^d}f(\bm{x}+2a\bm{k}),\quad\forall\bm{x}\in[-a,a]^d,
\end{equation}
where $\mathbb{Z}$ denotes the set of integers. Such formula for $\widetilde{f}$ can be found in \cite[page 196]{nodehi2021estimation}.  In this paper, we consider the density $f$ satisfying the exponential decay condition in \cite[equation (22)]{nuyens2023scaled}, which guarantees the uniform convergence on $[-a,a]^d$ of the series in \eqref{wrapped_f}. 

Notably, the periodic density function $\widetilde{f}$ can be interpreted as the wrapped version of $f$, a concept frequently employed in toroidal density estimation (see references \cite{agostinelli2007robust,di2011kernel,nodehi2021estimation}). Unlike existing approaches that passively estimate the wrapped density function on a given fixed torus, our method actively transforms the original density $f$ into its wrapped version $\widetilde{f}$ on an adaptive torus whose size scales with the sample size $M$, and then performs the density estimation on $\widetilde{f}$. By allowing the scaling parameter $a$ to grow with $M$, we ensure that $\widetilde{f}$ converges to $f$ with increasing accuracy. Specifically, if $f$ satisfies the exponential decay condition up to order $\alpha$  (see \cite[equation (22)]{nuyens2023scaled}), we prove that the discrepancy between $\widetilde{f}$ and $f$ decays exponentially with $a$. Based on this result we apply the approximation~(\ref{approx_tpye}) to $\widetilde{f}$ as follows:
find $\widetilde{f}^\lambda_{\widetilde{\bm{Y}}}\in V_N$ such that 
\begin{equation}\label{intro_PSKK}
    \left\langle\widetilde{f}^\lambda_{\widetilde{\bm{Y}}},v\right\rangle_{L^2([-a,a]^d)} + \lambda \left\langle \widetilde{f}^\lambda_{\widetilde{\bm{Y}}},v\right\rangle_{K_{\alpha,a,d}}=\frac{1}{M}\sum_{m=1}^{M}v(\widetilde{Y}_m),\quad\text{for all } v\in V_N,
\end{equation}
where $\langle\cdot,\cdot\rangle_{K_{\alpha,a,d}}$ denotes the inner product in the scaled Korobov space over $[-a,a]^d$, $\widetilde{\bm{Y}}:=\{\widetilde{Y}_1,\ldots,\widetilde{Y}_{M}\}$, $\lambda>0$ is a ``regularization'' parameter, and
\begin{equation*}
    V_N :=V_N(X):= \text{span}\{K_{\alpha,a,d}(\bm{x}_k,\cdot)|k=1,\ldots,N\},
\end{equation*}
with a preselected point set $X=\{\bm{x}_1,\ldots,\bm{x}_N\}\subset [-a,a]^d$. 
We then define the PSKK estimator
\begin{equation}
    \bar{f}(\bm{x})=f^{\lambda}_{\bm{Y}}(\bm{x}):=\begin{cases}
    \max\{\widetilde{f}^\lambda_{\widetilde{\bm{Y}}}(\bm{x}),0\}, &\text{ for }\bm{x}\in[-a,a]^d,\\
    0, &\text{ for }\bm{x}\in\mathbb{R}^d\setminus[-a,a]^d,
    \end{cases}\nonumber
\end{equation}
based on the sample set $\bm{Y}:=\{Y_1,\ldots,Y_M\}$ and establish that with appropriate choices of $a, N$ and $\lambda$, the MISE of the PSKK estimator achieves a convergence rate of $\mathcal{O}(|\log M|^{2(\alpha+1)d/q}M^{-1/(1+1/(2\alpha)+\epsilon)})$, where $q\ge 1$ is a parameter related to the decay of the density function and $\epsilon>0$ is arbitrarily small. Notably, although the exponent of the logarithmic factor depends on the dimension $d$, the MISE convergence rate can be made arbitrarily close to $\mathcal{O}(M^{-1})$  provided that $f$ is infinitely differentiable and satisfies the exponential decay condition.

Similar to \cite{kazashi2023density}, the PSKK estimator is not guaranteed to integrate to one. It is worth noting that relaxing the integral constraint is also considered in the context of standard KDE and can sometimes lead to higher convergence rates, as shown in \cite{terrell1980improving}.

The structure of this paper is organized as follows. Section~\ref{sec2} reviews fundamental concepts related to both the standard Korobov spaces and the scaled Korobov spaces \cite{nuyens2023scaled}, as well as providing a detailed revisit of  Kazashi and Nobile's estimator in \cite{kazashi2023density}. Section~\ref{sec3} focuses on density functions in the scaled Korobov spaces and obtains the MISE bound with the scaled parameter $a$. In Section~\ref{sec4}, we apply the density estimator to the wrapped density $\widetilde{f}$ derived from the target density function $f$, and obtain the total MISE bound under the exponential decay assumption. Section~\ref{sec5} discusses some details in implementing our PSKK method. Section~\ref{sec6} provides a portion of numerical experiments to support the theoretical results, while the other portion is presented in the appendix. Finally, in Section~\ref{sec7}, we draw the conclusions of the paper.

\section{Notations and backgrounds}\label{sec2}
Let $(D,\mathcal{B},\mu)$ be a measurable space. We denote by $L^2_{\mu}(D)$ the space of square-integrable functions with respect to the measure $\mu$ over $D$. Specifically, $L^2(D)$ refers to the special case where $D$ is a subset of $\mathbb{R}^d$ and $\mu$ is the Lebesgue measure.  Denote $\mathcal{H}(K)$ as the RKHS with reproducing kernel $K:D\times D\to \mathbb{R}$.
Let $\Vert f\Vert_K$ and $\left\langle f, g\right\rangle_K$ be the corresponding norm of $f\in\mathcal{H}(K)$ and inner product of $f$ and $g$ for $f,g\in\mathcal{H}(K)$. By the reproducibility of $\mathcal{H}(K)$ we have
\begin{equation}\label{reproducing_property}
    \left\langle f, K(\cdot,\bm{y})\right\rangle_K=f(\bm{y}),\quad\forall \bm{y}\in D \text{ and }\forall f\in\mathcal{H}(K).
\end{equation}

To denote the classical partial derivatives of $f$, we adopt the following notations
\begin{equation}
    f^{(\bm{\tau})}(\bm{x})=
    \prod_{j=1}^{d}\frac{\partial^{\tau_j}f}{\partial x_j^{\tau_j}} (\bm{x}),\quad \partial_{u}^{k}f(\bm{x})=\prod_{j\in u}\frac{\partial^{k}f}{\partial x_j^{k}} (\bm{x}),\quad \ \forall\bm{\tau}\in\mathbb{N}_0^d,k\in\mathbb{N}_0,u\subset\{1{:}d\},\nonumber
\end{equation}
where $\mathbb{N}_0$ denotes the set of non-negative integers.

For any $k,s\in \mathbb{Z}$ with $k\le s$, let $\{k{:}s\}=\{k,k+1,\ldots,s\}$. For any subset $u\subset \{1{:}d\}$, let $-u$ denote the complement of $u$, $\bm{y}_u$ denote the vector $(y_j)_{j\in u}$. Additionally, we denote $\bm{1}$ as the $d$-dimensional all-ones vector.

For any $\bm{a}=(a,\ldots,a)^\top\in (0,\infty)^d$, we denote by $\mathcal{J}_a$ the linear scaling  from $[0,1]^d$ to $[-a,a]^d$ given by 
\begin{equation}\label{linear_scaling}
    \mathcal{J}_a(\bm{x}):= 2a\bm{x}-\bm{a},\quad\forall\bm{x}\in[0,1]^d,
\end{equation}
and let $\mathcal{J}_a^{-1}$ be its inverse.

We also introduce the periodicity on $[0,1]^d$ and $[-a,a]^d$.
\begin{definition}
    For a function $h:[0,1]^d\to \mathbb{R}$, we say $h$ is periodic on $[0,1]^d$ if
    \begin{equation*}
        h(x_1,\ldots,x_{j-1},0,x_{j+1},\ldots,x_d)=h(x_1,\ldots,x_{j-1},1,x_{j+1},\ldots,x_d)
    \end{equation*}
    holds for any $\bm{x}\in [0,1]^d$ and any $1\le j\le d$. Furthermore, for $\alpha\in\mathbb{N}$, we say $h$ is $C^{(\alpha,\ldots,\alpha)}$-periodic on $[0,1]^d$ if for any $\bm{\tau}\in\{0{:}\alpha\}^d$, the derivative $h^{(\bm{\tau})}$ is continuous and periodic on $[0,1]^d$.
    
    Similarly, for $a>0$ and a function $h:[-a,a]^d\to \mathbb{R}$, we say $h$ is periodic (or $C^{(\alpha,\ldots,\alpha)}$-periodic) on $[-a,a]^d$ if $h\circ\mathcal{J}_a$ is  periodic (or $C^{(\alpha,\ldots,\alpha)}$-periodic) on $[0,1]^d$.
\end{definition}

\subsection{Korobov space on the unit cube}
This subsection focuses on periodic functions over
 the unit cube $[0,1]^d$ admitting an absolutely convergent Fourier series representation
\begin{equation}
    f(\bm{x})=\sum_{\bm{h}\in\mathbb{Z}^d}\widehat{f}(\bm{h})\exp(2\pi i \langle\bm{h},\bm{x}\rangle),\quad \widehat{f}(\bm{h}):=\int_{[0,1]^d} f(\bm{x})\exp(-2\pi i \langle\bm{h},\bm{x}\rangle)d\bm{x},\nonumber
\end{equation}
where $\langle\cdot,\cdot\rangle$ denotes the Euclidean inner product on $\mathbb{R}^d$. The Korobov space is an RKHS defined as follows.
\begin{definition}
    For $\alpha\in\mathbb{N}$, the space $\mathcal{H}(K_{\alpha,d})$ is a Korobov space on the unit cube $[0,1]^d$, which is an RKHS with the inner product
    \begin{equation}
        \left\langle f, g\right\rangle_{K_{\alpha,d}}:=\sum_{\bm{h}\in\mathbb{Z}^d}\widehat{f}(\bm{h})\overline{\widehat{g}(\bm{h})}r_{\alpha,d}(\bm{h})^2,\nonumber
    \end{equation}
    and the reproducing kernel 
    \begin{equation}
        K_{\alpha,d}(\bm{x},\bm{y}):=\sum_{\bm{h}\in\mathbb{Z}^d}\frac{\exp(2\pi i \langle\bm{h},(\bm{x}-\bm{y})\rangle)}{r_{\alpha,d}(\bm{h})^2}=\prod_{j=1}^{d}\left(1+(-1)^{\alpha+1}\frac{B_{2\alpha}(|x_j-y_j|)}{(2\alpha)!}\right),\nonumber
    \end{equation}
    where $B_{k}(x)$ is the Bernoulli polynomial of degree $k$ and
    \begin{equation}
        r_{\alpha,d}(\bm{h})=\prod_{j=1}^{d}r_{\alpha}(h_j),\quad r_{\alpha}(h_j)
        =\begin{cases}
            1, & \text{for } h_j = 0, \\ 
            |2\pi h_j|^{\alpha}, & \text{for } h_j \neq 0.
         \end{cases}\nonumber
    \end{equation}
\end{definition}

\begin{remark}\label{remark_Kor}
    The Korobov space that we consider here is a special case of the Korobov spaces in \cite{cools2020lattice, kazashi2023density}, obtained by setting their smoothness parameter to be $2\alpha$ and their weights $\gamma_u$ to be $(2\pi)^{-2\alpha|u|}$.
\end{remark}

It should be noted that $\mathcal{H}(K_{\alpha,d})$ contains all functions that are  $C^{(\alpha,\ldots,\alpha)}$-periodic on $[0,1]^d$, and the inner product of Korobov space $\mathcal{H}(K_{\alpha,d})$ can be expressed in the following form
\begin{align*}
  \left\langle f, g\right\rangle_{K_{\alpha,d}}=\sum_{u\subset\{1{:}d\}}\int_{[0,1]^{|u|}}\left(\int_{[0,1]^{d-|u|}}\partial_u^\alpha f(\bm{x})d\bm{x}_{-u}\right)\left(\int_{[0,1]^{d-|u|}}\partial_u^\alpha g(\bm{x})d\bm{x}_{-u}\right)d\bm{x}_u.
\end{align*}

\subsection{Korobov space on a box}

We now define the Korobov space on a box $[-a,a]^d\subset\mathbb{R}^d$ with $a>0$. This corresponds to the scaled Korobov space from \cite{nuyens2023scaled} with the domain specialized to the symmetric box $[-a,a]^d$. Firstly, we shall recall the `scaled' version of trigonometric polynomials in \cite{nuyens2023scaled}. For any $h\in\mathbb{Z}$, let
\begin{equation}
    \varphi_{a,h}(x)=\frac{1}{\sqrt{2a}}\exp\left(2\pi ih\frac{x+a}{2a}\right).\nonumber
\end{equation}
For $\bm{h}\in\mathbb{Z}^d$, let 
$\varphi_{a,d,\bm{h}}(\bm{x})=\prod_{j=1}^{d}\varphi_{a,h_j}(x_j)$. Then $\{\varphi_{a,d,\bm{h}}\}$ constitutes a standard orthonormal system in $L^2([-a,a]^d)$.

Now we consider the `periodic' space on the box $[-a,a]^d$, in which the function has the following form of absolutely converging Fourier
series with respect to $\varphi_{a,d,\bm{h}}$
\begin{equation}
    f(\bm{x})=\sum_{\bm{h}\in\mathbb{Z}^d}\widehat{f}_{a,d}(\bm{h})\varphi_{a,d,\bm{h}}(\bm{x}),\quad \widehat{f}_{a,d}(\bm{h})=\int_{[-a,a]^d}f(\bm{x})\overline{\varphi_{a,d,\bm{h}}(\bm{x})}d\bm{x}.\nonumber
\end{equation}

\begin{definition}
For $\alpha\in\mathbb{N}$, the space $\mathcal{H}(K_{\alpha,a,d})$ is a Korobov space on the box $[-a,a]^d$, which is an RKHS with the inner product
\begin{equation}
    \left\langle f, g\right\rangle_{K_{\alpha,a,d}}:=\sum_{\bm{h}\in\mathbb{Z}^d}\widehat{f}_{a,d}(\bm{h})\overline{\widehat{g}_{a,d}(\bm{h})}r_{\alpha,a,d}(\bm{h})^2,\nonumber
\end{equation}
and the reproducing kernel
\begin{align}
    K_{\alpha,a,d}(\bm{x},\bm{y}):=&\sum_{\bm{h}\in\mathbb{Z}^d}r_{\alpha,a,d}(\bm{h})^{-2}\varphi_{a,d,\bm{h}}(\bm{x})\overline{\varphi_{a,d,\bm{h}}(\bm{y})}\label{kernel_ser}\\
    =&\prod_{j=1}^{d}\left(\frac{1}{(2a)^2}+\frac{(-1)^{\alpha+1}(2a)^{2\alpha-1}}{(2\alpha)!}B_{2\alpha}\left(\frac{|x_j-y_j|}{2a}\right) \right),\nonumber
\end{align}
where
\begin{equation}\label{r}
   r_{\alpha,a,d}(\bm{h}):=\prod_{j=1}^{d}r_{\alpha,a}(h_j),\quad 
    r_{\alpha,a}(h_j)=\begin{cases}
        \sqrt{2a}, &\text{ for }h_j=0,\\
        |\pi h_j/a|^\alpha, &\text{ for }h_j\ne 0.
    \end{cases}
\end{equation}

\end{definition}

Similar to the Korobov space on the unit cube, the Korobov space $\mathcal{H}(K_{\alpha,a,d})$ contains all functions that are $C^{(\alpha,\ldots,\alpha)}$-periodic on $[-a,a]^d$, and the inner product of $\mathcal{H}(K_{\alpha,a,d})$ has the following form
\begin{align*}
&\left\langle f, g\right\rangle_{K_{\alpha,a,d}}\\
=& \sum_{u\subset\{1{:}d\}}\int_{[-a,a]^{|u|}}\left(\int_{[-a,a]^{d-|u|}}\partial_u^\alpha f(\bm{x})d\bm{x}_{-u}\right)\left(\int_{[-a,a]^{d-|u|}}\partial_u^\alpha g(\bm{x})d\bm{x}_{-u}\right)d\bm{x}_u.
\end{align*}

\subsection{Density estimation in RKHS}
In this subsection, with an abuse of notation, we recall the theory of density estimation in general RKHS as presented in \cite{kazashi2023density}. Let $(\Omega,\mathcal{F},\mathbb{P})$ be a probability space and $(D,\mathcal{B})$ be a measurable space, where $D$ is a general set. Suppose independent random vectors $Y_1,\ldots,Y_M:\Omega\to D$ follow an identical distribution defined by a density $f\in\mathcal{H}(K)$ with respect to a measure $\mu$ on $\mathcal{B}$, where $K:D\times D\to \mathbb{R}$ is a reproducing kernel. Thus the sample set $\bm{Y}:=\{Y_1,\ldots,Y_M\}$ consists of points drawn from $D$. The kernel method in \cite{kazashi2023density} consider the estimator
\begin{equation}\label{de_RKHS}
    f^\lambda_{\bm{Y}}(\cdot)= \sum_{k=1}^{N}c_k(\omega)K(x_k,\cdot),
\end{equation}
where $N$ is a positive integer, $X = \{x_1,\ldots,x_N\}$ is a preselected point set in $D$, and the (random) coefficient vector $\bm{c}(\omega):=(c_1(\omega),\ldots,c_N(\omega))^\top\in\mathbb{R}^N$ depends on the sample set $\bm{Y}(\omega)=\{ Y_1(\omega),\ldots,Y_M(\omega)\}$ and the ``regularization'' parameter $\lambda$. Specifically, if we denote 
\begin{equation}
    V_N :=V_N(X):= \text{span}\{K(x_k,\cdot)|k=1,\ldots,N\}\subset\mathcal{H}(K),\nonumber
\end{equation}
then the coefficient vector $\bm{c}(\omega)$ can be determined by specializing the variational problem~(\ref{approx_tpye}) to the RKHS setting as follows
\begin{equation}\label{estimation_RKHS}
\left\langle f_{\bm{Y}}^\lambda,v\right\rangle_{L^2_\mu(D)}+\lambda\left\langle f_{\bm{Y}}^\lambda,v\right\rangle_K=\frac{1}{M}\sum_{m=1}^{M}v(Y_m)\quad \text{for all $v\in V_N$}.
\end{equation}
Consistent with the situation of \cite{kazashi2023density}, it is equivalent to solving the linear system
\begin{equation}\label{linear system}
    A\bm{c}=\bm{b},
\end{equation}
where the matrix $A\in\mathbb{R}^{N\times N}$ is given by $A_{j,k}=\left\langle K(x_j,\cdot),K(x_k,\cdot)\right\rangle_{L^2_\mu(D)}+\lambda K(x_j,x_k)$ for $j,k=1,\ldots,N$, and the vector $\bm{b}\in\mathbb{R}^{N}$ is given by $b_j=\frac{1}{M}\sum_{m=1}^{M}K(x_j,Y_m)$ for $j=1,\ldots,N$. 
The solution to (\ref{linear system}) can be viewed as an interpolation of the function $b(x)$ at the preselected point set $X$, where $b(x):=\frac{1}{M}\sum_{m=1}^{M}K(x,Y_m)$ is the average of the kernel function.
Remarkably, the formula~(\ref{estimation_RKHS}) here is a special case of (\ref{approx_tpye}), serving as the foundation for our PSKK method.

In order to analyze the MISE convergence, the kernel $K$ is assumed to admit an $L^2_{\mu}(D)$ orthonormal expansion
\begin{equation}
K(x,y)=\sum_{l=0}^{\infty}\beta_l\varphi_l(x)\varphi_l(y),\nonumber
\end{equation}\label{kernel}
where $\{\varphi_l\}_{l=0}^{\infty}$ is an orthonormal system of $L^2_\mu(D)$ and $(\beta_l)_{l=0}^{\infty}$ is a positive sequence converging to $0$. With this assumption, the inner product for $\mathcal{H}(K)$ is given by 
\begin{equation}
    \left\langle f,g\right\rangle_K = \sum_{l=0}^{\infty}\beta_l^{-1}\left\langle f,\varphi_l\right\rangle_{L^2_\mu(D)}\left\langle g,\varphi_l\right\rangle_{L^2_\mu(D)}.\nonumber
\end{equation}
The continuum scale of nested Hilbert spaces related to $\mathcal{H}(K)$ is also introduced. For $\tau>0$, denote
\begin{equation}
    \mathcal{N}^\tau(K):=\left\{v\in L_\mu^2:\Vert v\Vert_{\mathcal{N}^\tau(K)}^2=\sum_{l=0}^{\infty}\beta_{l}^{-\tau}\left|\left<v,\varphi_l\right>_{L^2_\mu(D)}\right|^2<\infty\right\}.\nonumber
\end{equation}
Denote the normed space
\begin{equation}
    \mathcal{N}^{-\tau}(K):=\left\{\Psi:\mathcal{N}^\tau(K)\to\mathbb{R}, \Psi(v)=\sum_{l=0}^{\infty}\Psi_l\left<v,\varphi_l\right>_{L^2_\mu(D)}:\Vert \Psi\Vert_{\mathcal{N}^{-\tau}(K)}<\infty\right\},\nonumber
\end{equation}
where $\Vert \Psi\Vert_{\mathcal{N}^{-\tau}(K)}^2:=\sum_{l=0}^{\infty}\beta_l^\tau|\Psi_l|^2$. By \cite[proposition 2.1]{kazashi2023density}, $\mathcal{N}^{-\tau}(K)$ can be regarded as the dual space of $\mathcal{N}^{\tau}(K)$. These spaces form the foundation for our analysis in Section~\ref{sec3}, where we will use them to derive an MISE upper bound. Such upper bound is influenced by the parameter $\tau$, whose role is clarified below.

\begin{remark}
    The parameter $\tau$ governs the smoothness of the functions in $\mathcal{N}^{\tau}(K)$, with larger $\tau$ yielding smoother functions and a larger dual space $\mathcal{N}^{-\tau}(K)$. As we will see in Theorem \ref{de_kazashi}, the MISE convergence rate is influenced by how “small” this dual space can be while still containing the linear functionals that appear in (\ref{estimation_RKHS}). In other words, the convergence rate depends on the smallest $\tau$ for which these functionals belong to $\mathcal{N}^{-\tau}(K)$. This relationship manifests in the third term of the MISE bound (\ref{simple_MISE_bound}), where a smaller $\tau$ (implying a more restrictive dual space) leads to a higher asymptotic convergence rate.
\end{remark}

\section{Density estimation in Korobov spaces on a box}\label{sec3}
In this section we apply the estimation~(\ref{estimation_RKHS}) to the density function $f\in \mathcal{H}(K_{\alpha,a,d})$ with $a>0$. Under these settings, we have $D = [-a,a]^d$ and $\mu$ is the Lebesgue measure on $[-a,a]^d$. The orthogonal system $\{\varphi_l\}$ is replaced by $\{\varphi_{a,d,\bm{h}}\}$ and the coefficients $\{\beta_l\}$ are replaced by $\{r_{\alpha,a,d}(\bm{h})^{-2}\}$. Thus the preselected point set $X = \{\bm{x}_1,\ldots,\bm{x}_N\}\subset [-a,a]^d$ and 
\begin{equation}\label{sec3_V_N}
    V_N=V_N(X)= \text{span}\left\{K_{\alpha,a,d}(\bm{x}_k,\cdot):k=1,\ldots,N\right\}.
\end{equation}

Now let $P_N:\mathcal{H}(K_{\alpha,a,d})\to V_N$ be the $\mathcal{H}(K_{\alpha,a,d})$-orthogonal projection and denote
\begin{equation}\label{delta&F}
    \Delta_{\bm{Y}}(v):=\frac{1}{M}\sum_{m=1}^{M}v(Y_m),\quad F(v):=\int_{[-a,a]^d}v(\bm{x})f(\bm{x})d\bm{x}.
\end{equation}
We recall \cite[Theorem 3.6]{kazashi2023density} to bound the MISE.
\begin{theorem}\label{de_kazashi}
    Let $f\in\mathcal{H}(K_{\alpha,a,d})$ be the target density function and let $f_{\bm{Y}}^{\lambda}\in V_N$ satisfy (\ref{estimation_RKHS}). Suppose that for some $\tau\in(0,1]$ we have $\Delta_{\bm{Y}},F\in\mathcal{N}^{-\tau}(K_{\alpha,a,d})$ and that $f$ satisfies $\left\langle K_\tau(\cdot,\cdot),f\right\rangle_{L^2([-a,a]^d)}<\infty$ with 
    \begin{equation}\label{k_tau}
       K_\tau(\bm{x},\bm{y})=\sum_{\bm{h}\in\mathbb{Z}^d}r_{\alpha,a,d}(\bm{h})^{-2\tau}
        \varphi_{a,d,\bm{h}}(\bm{x})\overline{\varphi_{a,d,\bm{h}}(\bm{y})},
    \end{equation}
    where $r_{\alpha,a,d}(\bm{h})$ is defined by (\ref{r}). Then we have the MISE bound
    \begin{align}\label{simple_MISE_bound}
    &\mathbb{E}\left[\int_{[-a,a]^d}\left|f^\lambda_{\bm{Y}}(\bm{x})-f(\bm{x})\right|^2d\bm{x}\right]\nonumber\\
     \le& \Vert P_N f-f\Vert_{L^2([-a,a]^d)}^2+\lambda\Vert f\Vert_{K_{\alpha,a,d}}^2 + \frac{\left<K_\tau(\cdot,\cdot),f\right>_{L^2([-a,a]^d)}}{M\lambda^\tau}.
    \end{align}
\end{theorem}

To get a bound on MISE, it remains to determine the range of $\tau$ and to handle the term $\Vert P_N f-f\Vert_{L^2([-a,a]^d)}$.

\subsection{The lower bound for $\tau$}\label{lower_tau}
We first derive a lower bound for $\tau$ such that the conditions in Theorem~\ref{de_kazashi} are satisfied. To establish this bound, we require the following three technical lemmas. The proofs of these lemmas are analogous to
the proof of \cite[Proposition 4.2]{kazashi2023density}, with the additional observation that $\varphi_{a,d,\bm{h}}(\bm{x})\le (2a)^{-\frac{d}{2}},\forall \bm{x}\in[-a,a]^d$ and that $f$ is a density function on $[-a,a]^d$. We provide these proofs in detail in Appendix~\ref{SM_point_evaluation_functional}, \ref{SM_F} and \ref{SM_L2_K_tau}.

\begin{lemma}\label{point_evaluation_functional}
    If $\tau>\frac{1}{2\alpha}$, then for any $\bm{x}\in[-a,a]^d$, the point evaluation functional $\delta_{\bm{x}}(v)=v(\bm{x})$ satisfies $\delta_{\bm{x}}\in\mathcal{N}^{-\tau}(K_{\alpha,a,d})$.
\end{lemma}

\begin{lemma}\label{F}
    If $\tau>\frac{1}{2\alpha}$, then $F\in \mathcal{N}^{-\tau}(K_{\alpha,a,d})$. 
\end{lemma}

\begin{lemma}\label{L2_K_tau}
    For $\tau>\frac{1}{2\alpha}$, let $K_\tau(\bm{x},\bm{y})$ be defined in (\ref{k_tau}), then for density function $f\in\mathcal{H}(K_{\alpha,a,d})$ we have 
    \begin{equation*}
        \left\langle K_\tau(\cdot,\cdot),f(\cdot)\right\rangle_{L^2([-a,a]^d)}=\left((2a)^{-\tau-1}+\frac{1}{a}\left(\frac{a}{\pi}\right)^{2\alpha\tau}\zeta(2\alpha\tau)\right)^{d}<\infty,
    \end{equation*}
\end{lemma}
where $\zeta(s):=\sum_{n=1}^\infty n^{-s}$ for $s>1$ is the Riemann zeta function.

Therefore, the conditions in Theorem~\ref{de_kazashi} hold if $\tau>1/(2\alpha)$.

\subsection{Error bound for the projection}\label{sub_projection}
In this subsection we focus on the $L^2$-error between $f$ and $P_Nf$. Since $K_{\alpha,a,d}$ is a reproducing kernel, an important observation is that the projection $P_Nf$ coincides with the kernel interpolation (see, for example, \cite{deboor1966splines}).
This follows from the orthogonality of $P_N$ and the reproducing property (\ref{reproducing_property}). Indeed, for any $1\le k\le N$, we have
\begin{equation*}
    0 = \left\langle f-P_Nf, K_{\alpha,a,d}(\bm{x}_k,\cdot)\right\rangle_{K_{\alpha,a,d}}=f(\bm{x}_k)-P_Nf(\bm{x}_k).
\end{equation*}
We note that $\mathcal{H}(K_{\alpha,a,d})\subset L^2([-a,a]^d)$ since all functions in $\mathcal{H}(K_{\alpha,a,d})$ are continuous and $[-a,a]^d$ is a finite interval. Combining this observation with the fact that the projection $P_Nf$ coincides with the kernel interpolation allows us to apply \cite[Theorem 2.2]{kaarnioja2022fast}, which yields the following lemma regarding the $L^2$-optimality of $P_N$.

\begin{lemma}\label{op_P}
    Let $A_N:\mathcal{H}(K_{\alpha,a,d})\to \mathcal{H}(K_{\alpha,a,d})$ be an algorithm  such that $A_N(\phi)=A_N(\phi(\bm{x}_1),\ldots,\phi(\bm{x}_N))$ only depends on $\phi(\bm{x}_1),\ldots,\phi(\bm{x}_N)$ for all $\phi\in\mathcal{H}(K_{\alpha,a,d})$. Then
    \begin{equation*}
        \sup_{\Vert\phi\Vert_{K_{\alpha,a,d}}\le 1}\Vert P_N\phi-\phi\Vert_{L^2([-a,a]^d)}\le \sup_{\Vert\phi\Vert_{K_{\alpha,a,d}}\le 1}\Vert A_N(\phi)-\phi\Vert_{L^2([-a,a]^d)}
    \end{equation*}
\end{lemma}

Based on this lemma, we only need to determine the preselected point set $X$ and an algorithm $A_N$ depending on $\bm{x}_1,\ldots,\bm{x}_N$ such that $\Vert A_N(\phi)-\phi\Vert_{L^2([-a,a]^d)}$ converges to $0$ rapidly. To construct such an algorithm,  we first reduce the approximation problem from $\mathcal{H}(K_{\alpha,a,d})$ over $[-a,a]^d$ to the Korobov space $\mathcal{H}(K_{\alpha,d})$ over $[0,1]^d$. This is achieved by mapping each function 
$\phi \in \mathcal{H}(K_{\alpha,a,d})$ to $\phi\circ\mathcal{J}_a$, where  $\mathcal{J}_a$ is the linear scaling defined in (\ref{linear_scaling}). We then perform the lattice algorithm~\cite{cools2020lattice} in $\mathcal{H}(K_{\alpha,d})$, which relies solely on function values at lattice points and offers an efficient approximation in general Korobov spaces over $[0,1]^d$. Since $\mathcal{H}(K_{\alpha,d})$ is a special case of Korobov spaces over $[0,1]^d$, we can apply \cite[Theorem 3.6]{cools2020lattice} with the corresponding parameters specified in Remark~\ref{remark_Kor}. This yields the following lemma regarding the $L^2$-error for the lattice algorithm in $\mathcal{H}(K_{\alpha,d})$.

\begin{lemma}\label{L2_B_N}
    Given integers $d$, $\alpha$, and a prime $N$, a generating vector $\bm{z}^*$ can be constructed by the CBC construction in \cite{cools2020lattice} so that for 
    \begin{equation}
        Z^*=\left\{\bm{y}^*_k=\left\{\frac{k\bm{z}^*}{N}\right\}:k=1,\ldots,N\right\},\nonumber
    \end{equation}
    where $\{\cdot\}$ denotes the fractional part of each component,
    the corresponding lattice algorithm $B_N^*:\mathcal{H}(K_{\alpha,d})\to \mathcal{H}(K_{\alpha,d})$ in \cite{cools2020lattice} uses only function values at $Z^*$ and satisfies
    \begin{equation}
        \Vert h-B_N^*(h)\Vert_{L^2([0,1]^d)}\le C_{\alpha,\delta,d}\Vert h\Vert_{K_{\alpha,d}}N^{-(\alpha/2-\delta)},\quad\forall h\in\mathcal{H}(K_{\alpha,d}),\delta\in (0,\alpha/2),\nonumber
    \end{equation}
    where 
    \begin{equation}\label{constant_C}
        C_{\alpha,\delta,d}= \kappa_{\alpha,\delta}  \left(\sum_{u\subset \{1{:}d\}}\max(|u|,1)\left[2\left(\frac{1}{2\pi}\right)^{\frac{\alpha}{\alpha-2\delta}}\zeta\left(\frac{\alpha}{\alpha-2\delta}\right)\right]^{|u|}\right)^{\alpha-2\delta},
    \end{equation}
    with 
    \begin{equation*}
       \kappa_{\alpha,\delta}=\sqrt{2}\max\left\{6,2.5+2^{\frac{2\alpha}{\alpha-2\delta}+1}\right\}^{\alpha/2-\delta}.
    \end{equation*}
\end{lemma}
\begin{remark}
    The constant $C_{\alpha,\delta,d}$ in (\ref{constant_C}) is not identical to the one in \cite[Theorem 3.6]{cools2020lattice}, because we now optimize over the truncation parameter in \cite{cools2020lattice} for the lattice algorithm. For a similar treatment, see \cite[Theorem 3.2]{kaarnioja2022fast}.
\end{remark}

Now for function $\phi\in\mathcal{H}(K_{\alpha,a,d})$, we consider  $h(\bm{y}):=\phi\circ\mathcal{J}_a(\bm{y})=\phi(2a\bm{y}-\bm{a})$.
Then $h\in\mathcal{H}(K_{\alpha,d})$. For any $a\ge\frac{1}{2}$, we have
\begin{align}
\Vert h\Vert_{K_{\alpha,d}}^2&=\sum_{u\subset\{1{:}d\}}\int_{[0,1]^{|u|}}\left(\int_{[0,1]^{d-|u|}}\partial_u^\alpha
 h(\bm{y}) d\bm{y}_{-u}\right)^2d\bm{y}_{u}\nonumber\\
    &=\sum_{u\subset\{1{:}d\}}(2a)^{(2\alpha+1)|u|-2d}\int_{[-a,a]^{|u|}}
    \left(\int_{[-a,a]^{d-|u|}}\partial_u^\alpha\phi(\bm{x}) d\bm{x}_{-u}\right)^2d\bm{x}_{u}\nonumber\\
    &\le (2a)^{(2\alpha-1)d}\Vert \phi\Vert_{K_{\alpha,a,d}}^2,   \label{norm_h} 
\end{align}
where in the last inequality we use $2a\ge 1$. We define an algorithm
\begin{equation}\label{A_N}
    A_N^*(\phi):=B_N^*(\phi\circ\mathcal{J}_a)\circ\mathcal{J}_a^{-1},
\end{equation}
where $B_N^*$ is the lattice algorithm in Lemma~\ref{L2_B_N}. Then $A_N^*$ is an algorithm over $\mathcal{H}(K_{\alpha,a,d})$ that uses only function values on the scaled lattice point set
\begin{equation}\label{X*}
    X^*:=\mathcal{J}_a(Z^*)=\{\mathcal{J}_a(\bm{y}):\bm{y}\in Z^*\},
\end{equation}
where $Z^*$ is the lattice point set used in lattice algorithm $B_N^*$.
\begin{proposition}\label{L_2_A_N}
    Let $A_N^*$ and $X^*$ be defined as (\ref{A_N}) and (\ref{X*}), respectively. For any  $a\ge \frac{1}{2}$, $\phi\in\mathcal{H}(K_{\alpha,a,d})$ and any $\delta\in(0,\alpha/2)$, we have
    \begin{equation}
        \Vert\phi-A_N^*(\phi)\Vert_{L^2([-a,a]^d)}
        \le C_{\alpha,\delta,d}\Vert \phi\Vert_{K_{\alpha,a,d}}(2a)^{\alpha d}N^{-\alpha/2+\delta},\nonumber
    \end{equation}
    where $C_{\alpha,\delta,d}$ is defined by (\ref{constant_C}).
\end{proposition}

The proof of Proposition~\ref{L_2_A_N}  relies on Lemma~\ref{L2_B_N} and inequality~(\ref{norm_h}). We provide the details in Appendix~\ref{SM_proof_L_2_A_N}.

Combining Lemma~\ref{op_P} and Proposition~\ref{L_2_A_N}, we have the following bound for the projection error. This error bound can also be interpreted as the error bound for kernel interpolation using the scaled lattice point set $X^*$ as the interpolation nodes.
\begin{proposition}\label{P^K_N}
    Let $N$ be a prime number and let $V_N=V_N(X^*)$ with $X^*$ defined by (\ref{X*}). For any $a\ge \frac{1}{2}$,  $\phi\in\mathcal{H}(K_{\alpha,a,d})$ and any $\delta\in(0,\alpha/2) $, we have
\begin{equation}
    \Vert P_N\phi-\phi\Vert_{L^2([-a,a]^d)}\le C_{\alpha,\delta,d}\Vert \phi\Vert_{K_{\alpha,a,d}} (2a)^{\alpha d}N^{-(\alpha/2-\delta)}.\nonumber
\end{equation}
\end{proposition}

By incorporating the results of Lemma~\ref{L2_K_tau} and Proposition~\ref{P^K_N} into Theorem~\ref{de_kazashi}, we obtain an upper bound for the MISE.
\begin{theorem}\label{MISE_Kor}
    Given $d,\alpha\in\mathbb{N}$ and $a\ge \frac{1}{2}$. Assume that the target density $f\in\mathcal{H}(K_{\alpha,a,d})$ . Let $N$ be a prime number and let $V_N=V_N(X^*)$ with $X^*$ defined by (\ref{X*}). Let $f_{\bm{Y}}^\lambda\in V_N(X^*)$ satisfy (\ref{estimation_RKHS}). For any $\tau\in(1/(2\alpha),1],\lambda\in(0,1)$ and $\delta\in(0,\alpha/2)$, we have
    \begin{equation}
    \begin{aligned}
        &\mathbb{E}\left[\int_{[-a,a]^d}\left|f_{\bm{Y}}^\lambda(\bm{x})-f(\bm{x})\right|^2d\bm{x}\right]\\
     \le& C^2_{\alpha,\delta,d}\frac{(2a)^{2\alpha d}\Vert f\Vert^2_{K_{\alpha,a,d}} }{ N^{\alpha-2\delta}}+\lambda \Vert f\Vert_{K_{\alpha,a,d}}^2+ \frac{\left((2a)^{-\tau-1}+\frac{1}{a}\left(\frac{a}{\pi}\right)^{2\alpha\tau}\zeta(2\alpha\tau)\right)^{d}}{M\lambda^\tau},\nonumber
    \end{aligned}
    \end{equation}
    where $C_{\alpha,\delta,d}$ is defined in (\ref{constant_C}).
\end{theorem}

\begin{remark}
    Theorem~\ref{MISE_Kor} establishes the MISE bound for the estimator when the density function $ f $ lies in the scaled Korobov space, explicitly characterizing its dependence on $ N $, $ \lambda $, $ M $, and the scaling parameter $ a $. In Section~\ref{sec4}, we will transform the density function $f$ on $ \mathbb{R}^d $ into its wrapped version $ \widetilde{f} $ within the scaled Korobov space, thereby deriving the corresponding MISE bound via Theorem~\ref{MISE_Kor}.
\end{remark}

\section{Density estimation on $\mathbb{R}^d$}\label{sec4}
In this section, we consider the density function $f$ defined over $\mathbb{R}^d$. More precisely, we focus on functions with continuous mixed partial derivatives satisfying the following exponential decay condition \cite[equation (22)]{nuyens2023scaled}.

\begin{definition}[Exponential decay condition up to order $\alpha$]\label{exp_decay}
    For any function $h:\mathbb{R}^d\to \mathbb{R}$ with continuous mixed partial
    derivatives up to order $\alpha$ in each variable, we say $h$ satisfies the exponential decay condition up to order $\alpha$ if 
    \begin{equation*}
        \Vert h\Vert_{\alpha,\beta,p,q}:=\sup_{\bm{x}\in\mathbb{R}^d,\bm{\tau}\in\{0{:}\alpha\}^d}\left|\exp(\beta\Vert\bm{x}\Vert_{p}^q)h^{(\bm{\tau})} (\bm{x})\right|<\infty,
    \end{equation*}
    for some $\beta>0,1\le p\le\infty$ and $1\le q<\infty$, where $\Vert\bm{x}\Vert_{p}$ is the $\ell_p$-norm of $\bm{x}$.
\end{definition}

\begin{definition}
    Given $\alpha\in\mathbb{N},\beta>0,1\le p\le\infty$ and $1\le q<\infty$, let $S^\alpha_{\beta,p,q}$ be the collection of all function $f$ that has continuous  mixed partial derivatives up to order $\alpha$ in each variable such that $\Vert f\Vert_{\alpha,\beta,p,q}<\infty$ and the norm
    \begin{equation}
        \Vert f\Vert_{\alpha}^2:=\sum_{u\subset\{1{:}d\}}
        \int_{\mathbb{R}^{|u|}}\left(\int_{\mathbb{R}^{d-|u|}}|\partial_u^\alpha f(\bm{x})|d\bm{x}_{-u}\right)^2d\bm{x}_{u},\nonumber
    \end{equation}
    is finite.    
\end{definition}

\begin{remark}
The function space $S^{\alpha}_{\beta,p,q}$ includes common density functions such as Gaussian mixtures (with $q=2$) and logistic distributions (with $q=1$), where $\alpha$
can be any positive integer for both.
\end{remark}

Now we propose our PSKK method for density function $f\in S^\alpha_{\beta,p,q}$ as follows. Given $M$ independent random vectors $Y_1,\ldots,Y_M$ whose density function is $f$, we consider an unspecified parameter $a>0$ and denote $\widetilde{Y}_m$ as the remainder of $Y_m$ modulo $2a$ for all $1\le m \le M$. More specifically, for each $1\le m \le M$, we transform the random vector $Y_m=(Y_{m,1},\ldots,Y_{m,d})^\top$ into $\widetilde{Y}_m=(\widetilde{Y}_{m,1},\ldots,\widetilde{Y}_{m,d})^\top$ by defining
\begin{equation}
    \widetilde{Y}_{m,j}= (Y_{m,j}\ \text{mod $2a$)}-a\in [-a,a),\nonumber
\end{equation}
for all $1\le j\le d$. Then $\widetilde{Y}_1,\ldots,\widetilde{Y}_M$ are iid random vectors taking values in $[-a,a)^d$. The  density function of
$\widetilde{Y}_1$, denoted as $\widetilde{f}$, and its partial derivatives are characterized in the following proposition.

\begin{proposition}\label{prop_wtilde_f}
    The density function of $\widetilde{Y}_1$ is given pointwise  by
    \begin{equation}\label{wtilde_f}
    \widetilde{f}(\bm{x})=\sum_{\bm{k}\in\mathbb{Z}^d}f(\bm{x}+2a\bm{k}),\quad\forall\bm{x}\in[-a,a]^d,
    \end{equation}
    with the convention that the series may diverge to $+\infty$. Furthermore, if the original density $f$ satisfies the exponential decay condition up to order $\alpha$, then for any multi-index $\bm{\tau}\in\{0{:}\alpha\}^d$, the partial derivative $\widetilde{f}^{(\bm{\tau})}$ exists and is represented by the uniformly convergent series
    \begin{equation}\label{tilde_f_tau}
    \widetilde{f}^{(\bm{\tau})}(\bm{x})=\sum_{\bm{k}\in\mathbb{Z}^d}f^{(\bm{\tau})}(\bm{x}+2a\bm{k}),\quad \forall\bm{x}\in[-a,a]^d.
    \end{equation}
\end{proposition}

\begin{proof}
    We first prove (\ref{wtilde_f}) using a method analogous to that in \cite[page 196]{nodehi2021estimation}. According to the definition of $\widetilde{Y}_1$, for any measurable set $B\subset[-a,a)^d$, we have
    \begin{equation*}
        \mathbb{P}(\widetilde{Y}_1\in B)=\mathbb{P}\left(Y_1\in \bigcup_{\bm{k}\in \mathbb{Z}^d}(B+2a\bm{k}) \right),
    \end{equation*}
    where $B+2a\bm{k}=\{\bm{x}+2a\bm{k}:\bm{x}\in B\}$ denotes the set $B$ translated by the vector $2a\bm{k}$. Since $B\subset[-a,a)^d$, these translated sets are pairwise disjoint. So we have
    \begin{align}
        \mathbb{P}(\widetilde{Y}_1\in B)&=\sum_{\bm{k}\in \mathbb{Z}^d}\mathbb{P}(Y_1\in B+2a\bm{k})=\sum_{\bm{k}\in \mathbb{Z}^d}\int_{B+2a\bm{k}}f(\bm{x})d\bm{x}\nonumber\\
        &=\sum_{\bm{k}\in \mathbb{Z}^d}\int_{B}f(\bm{x}+2a\bm{k})d\bm{x}
        =\int_B \left(\sum_{\bm{k}\in \mathbb{Z}^d} f(\bm{x}+2a\bm{k})\right)d\bm{x},\label{half_cube}
    \end{align}
    where the last equation uses Tonelli’s theorem, since $f$ is non-negative. Equation~(\ref{half_cube}) shows that (\ref{wtilde_f}) holds for all $\bm{x}\in[-a,a)^d$. Note that $\widetilde{f}$ is a density function, we can extend it pointwise to the closed cube $[-a,a]^d$ using the same formula, which yields (\ref{wtilde_f}) on $[-a,a]^d$.

    Next, we prove the uniform convergence of the series in (\ref{tilde_f_tau}). Since $\Vert f\Vert_{\alpha,\beta,p,q}<\infty$, for any $\bm{\tau}\in\{0{:}\alpha\}^d,\bm{k}\in\mathbb{Z}^d\setminus\{\bm{0}\}$ and any $\bm{x}\in[-a,a]^d$, we have
    \begin{align*}
    |f^{(\bm{\tau})}(\bm{x}+2a\bm{k})|&\le \Vert f\Vert_{\alpha,\beta,p,q}\exp{(-\beta\Vert\bm{x}+2a\bm{k} \Vert_p^q)}\\
    &\le \Vert f\Vert_{\alpha,\beta,p,q}\exp{(-\beta\Vert\bm{x}+2a\bm{k} \Vert_\infty^q)}\\
    &\le \Vert f\Vert_{\alpha,\beta,p,q}\exp{(-\beta (2a\Vert\bm{k}\Vert_\infty-a)^q)},
    \end{align*}
    where the last inequality uses the fact that $\Vert \bm{x}+2a\bm{k}\Vert_\infty\ge 2a\Vert\bm{k\Vert}_\infty-a$, $\forall\bm{x}\in[-a,a]^d$. Note that
    \begin{equation*}
   \sum_{\bm{k}\in\mathbb{Z}^d\setminus\{\bm{0}\} }\exp{(-\beta (2a\Vert\bm{k}\Vert_\infty-a)^q)}<\infty. 
    \end{equation*}
    So the series in (\ref{tilde_f_tau}) converges uniformly. Using this uniform convergence together with (\ref{wtilde_f}), we can prove by induction that $\widetilde{f}^{(\bm{\tau})}$ exists and has the form given in (\ref{tilde_f_tau}).
\end{proof}

Proposition~\ref{prop_wtilde_f} shows that for $f\in S^\alpha_{\beta,p,q}$ and any $\bm{\tau}\in\{0{:}\alpha\}^d$, $\widetilde{f}^{(\bm{\tau})}$ is continuous and periodic on $[-a,a]^d$. Hence, $\widetilde{f}$ is $C^{(\alpha,\ldots,\alpha)}$-periodic on $[-a,a]^d$, and therefore belongs to the scaled Korobov space $\mathcal{H}(K_{\alpha,a,d})$.

\begin{remark}
    Alternatively, to transform $f$ into a periodic density function on $[-a,a]^d$, one may define a smooth transformation $T:[-a,a]^d\to\mathbb{R}^d$ 
    such that $Y_1=T(Z_1)$ for some random vector $Z_1$. The density of $Z_1$ is then given by $g(\bm{z}) = f(T(\bm{z})) \left| \det(DT(\bm{z})) \right|$, where $DT(\bm{z})$ is the Jacobian of $T$. Although periodization in quasi-Monte Carlo methods (e.g., polynomial transformations \cite{korobov1963number} or
    trigonometric transformations \cite{laurie1996periodizing,sidi1993new}) can construct $T$ to yield periodic $g$ for densities supported on $[0,1]^d$, extending such smooth transformations to $\mathbb{R}^d$ is hindered by the unbounded domain, rendering $T$ practically infeasible.  
    Consequently, we consider the modulo operation to periodize the density function, which, despite losing partial sample information, offers greater operational simplicity.
\end{remark}

We perform the density estimation~(\ref{estimation_RKHS}) on $\widetilde{f}$ in the  Korobov space $\mathcal{H}(K_{\alpha,a,d})$ and obtain an estimator $\widetilde{f}^\lambda_{\widetilde{\bm{Y}}}$, where $\widetilde{\bm{Y}}:=\{\widetilde{Y}_1,\ldots,\widetilde{Y}_{M}\}$. Finally, the PSKK estimator of $f$ is set to be
\begin{equation}\label{estimation_R^d}
f^{\lambda}_{\bm{Y}}(\bm{x}):=\begin{cases}
    \max\{\widetilde{f}^\lambda_{\widetilde{\bm{Y}}}(\bm{x}),0\}, &\text{ for }\bm{x}\in[-a,a]^d,\\
    0, &\text{ for }\bm{x}\in\mathbb{R}^d\setminus[-a,a]^d.
    \end{cases}
\end{equation}

\begin{remark}

The estimator~(\ref{estimation_R^d}) is constructed by applying the positive part operator $\max\{\cdot,0\}$ to ensure non-negativity. Since the target density $f$ is non-negative, this operation also reduces the MISE. However, the operator $\max\{\cdot,0\}$ may introduce kinks. 
Alternatively, one may consider a similar estimator that omits the $\max\{\cdot,0\}$ operator.
This estimator is smoother but may take negative values. We note that the convergence rate established in this paper applies to both estimators. In this paper, we place a greater emphasis on the global MISE and non-negativity than on smoothness, and we use the estimator~(\ref{estimation_R^d}). In contexts where smoothness is required, the estimator without the $\max\{\cdot,0\}$ operator may be preferred.

\end{remark}

For the MISE of the PSKK estimator $f^{\lambda}_{\bm{Y}}$, we have 
\begin{align}
&\mathbb{E}\left[\int_{\mathbb{R}^d}\left|f_{\bm{Y}}^\lambda(\bm{x})-f(\bm{x})\right|^2d\bm{x}\right]\nonumber\\
\le& \mathbb{E}\left[\int_{[-a,a]^d}\left|\widetilde{f}_{\widetilde{\bm{Y}}}^\lambda(\bm{x})-f(\bm{x})\right|^2d\bm{x}\right]+\int_{\mathbb{R}^d\setminus[-a,a]^d}|f(\bm{x})|^2d\bm{x}\nonumber\\
\le&2\Vert f-\widetilde{f}\Vert^2_{L^2([-a,a]^d)}+ 2\mathbb{E}\left[\Vert\widetilde{f}^\lambda_{\widetilde{\bm{Y}}}- \widetilde{f}\Vert^2_{L^2([-a,a]^d)}\right] + \int_{\mathbb{R}^d\setminus[-a,a]^d}|f(\bm{x})|^2d\bm{x}.\label{MISE_Rd}
\end{align}
Here, the first term on the right-hand side of (\ref{MISE_Rd}) measures the $L^2$-distance between $f$ and $\widetilde{f}$. The second term quantifies the MISE of estimating $\widetilde{f}$ in the scaled Korobov space. The third term corresponds to the truncation error caused by restricting the support of $f^\lambda_{\bm{Y}}$ to $[-a,a]^d$. Before bounding these three terms individually, we prove a key technical lemma.

\begin{lemma}\label{perodic_dif}
Suppose a function $h$ satisfies the exponential decay condition up to order $0$, i.e., $\Vert h\Vert_{0,\beta,p,q}=\sup_{\bm{x}\in\mathbb{R}^d}|\exp(\beta\Vert\bm{x}\Vert_{p}^q )h(\bm{x})|<\infty$, for some $\beta>0,1\le p\le\infty$ and $1\le q<\infty$. For $a>0$, let $\widetilde{h}(\bm{x})=\sum_{\bm{k}\in\mathbb{Z}^d}h(\bm{x}+2a\bm{k}),\ \forall \bm{x}\in[-a,a]^d$. Write
\begin{equation}\label{C_dq}
     C_{d,q}:=\sum_{w = 1}^{\infty}\left((2w+1)^d-(2w-1)^d\right)e^{-\left((2w-1)^q-1\right)}<\infty.
\end{equation}
Then for any any $a\ge \beta^{-1/q}$ and $\bm{x}\in[-a,a]^d$, the following inequality holds
    \begin{equation}\label{per_error}
        |h(\bm{x})-\widetilde{h}(\bm{x})|\le C_{d,q} \Vert h\Vert_{0,\beta,p,q} e^{-\beta a^q}.
    \end{equation}
\end{lemma}

\begin{proof}
For any $\bm{x}\in[-a,a]^d$, we have 
\begin{align}
    |h(\bm{x})-\widetilde{h}(\bm{x})|\le&\sum_{\bm{k}\in\mathbb{Z}^d\setminus\{\bm{0}\}}|h(\bm{x}+2a\bm{k})|\nonumber\\
    \le& \Vert h\Vert_{0,\beta,p,q}\sum_{\bm{k}\in\mathbb{Z}^d\setminus\{\bm{0}\}}\exp{(-\beta\Vert\bm{x}+2a\bm{k} \Vert_p^q)}\nonumber\\
    \le& \Vert h\Vert_{0,\beta,p,q}\sum_{\bm{k}\in\mathbb{Z}^d\setminus\{\bm{0}\}}\exp{(-\beta\Vert\bm{x}+2a\bm{k} \Vert_{\infty}^q)}\nonumber\\
   \le & \Vert h\Vert_{0,\beta,p,q}\sum_{w = 1}^{\infty}\sum_{\Vert \bm{k}\Vert_{\infty}=w}\exp(-\beta[(2w-1)a]^q)\label{infty_norm}\\
    =&\Vert h\Vert_{0,\beta,p,q}e^{-\beta a^q}\sum_{w = 1}^{\infty}\left((2w+1)^d-(2w-1)^d\right)e^{-\beta a^q\left((2w-1)^q-1\right)}\nonumber\\
    \le&\Vert h\Vert_{0,\beta,p,q}e^{-\beta a^q}\sum_{w = 1}^{\infty}\left((2w+1)^d-(2w-1)^d\right)e^{-\left((2w-1)^q-1\right)}\label{ineq_in_lemma}\\
    =& C_{d,q} \Vert h\Vert_{0,\beta,p,q} e^{-\beta a^q}.\nonumber
\end{align}
Here in (\ref{infty_norm}), we sum over $\bm{k}\in\mathbb{Z}^d\setminus\{\bm{0}\}$ by grouping the terms based on the values of $\Vert\bm{k\Vert}_\infty$, and note that for $\bm{x}\in[-a,a]^d$, $\Vert \bm{x}+2a\bm{k}\Vert_\infty\ge 2a\Vert\bm{k\Vert}_\infty-a$. In (\ref{ineq_in_lemma}), we apply the condition $\beta a^q\ge 1$. Hence, we obtain (\ref{per_error}).
\end{proof}

Lemma~\ref{perodic_dif} yields upper bounds for $\Vert f-\widetilde{f}\Vert^2_{L^2([-a,a]^d)}$ and $\Vert \widetilde{f}\Vert_{K_{\alpha,a,d}}$.

\begin{corollary}\label{term1}
Suppose $f\in S^{\alpha}_{\beta,p,q}$ for some $\alpha\in\mathbb{N},\beta>0,1\le p\le\infty$ and $1\le q<\infty$. Let $\widetilde{f}$ be defined in (\ref{wtilde_f}) and let $C_{d,q}$ be the constant defined in (\ref{C_dq}). For any $a\ge \beta^{-1/q}$, we have 
\begin{equation}
    \Vert f-\widetilde{f}\Vert^2_{L^2([-a,a]^d)}\le C_{d,q}^2\Vert f\Vert_{0,\beta,p,q}^2(2a)^d e^{-2\beta a^q}.\nonumber
\end{equation}
\end{corollary}

\begin{corollary}\label{term2}
Suppose $f\in S^{\alpha}_{\beta,p,q}$ for some $\alpha\in\mathbb{N},\beta>0,1\le p\le\infty$ and $1\le q<\infty$. Let $\widetilde{f}$ be defined in (\ref{wtilde_f}) and let $C_{d,q}$ be the constant defined in (\ref{C_dq}). For any $a\ge\max\{\beta^{-1/q},\frac{1}{2} \}$, we have 
\begin{equation}\label{wide_norm}
    \Vert \widetilde{f}\Vert_{K_{\alpha,a,d}}^2\le 2\Vert f\Vert_\alpha^2+ 2C_{d,q}^2\Vert f\Vert_{\alpha,\beta,p,q}^2(2\sqrt{2}a)^{2d}e^{-2\beta a^q}.
\end{equation}
\end{corollary}

\begin{proof}
Let $g(\bm{x}):=\widetilde{f}(\bm{x})-f(\bm{x})$ for $\bm{x}\in[-a,a]^d$. According to Lemma~\ref{perodic_dif}, for any $a\ge \beta^{-1/q}  $, $\bm{\tau}\in\{0{:}\alpha\}^d$, and $\bm{x}\in[-a,a]^d$, we have 
\begin{equation}\label{partial_g}
    |g^{(\bm{\tau})}(\bm{x})|\le C_{d,q}\Vert f^{(\bm{\tau})}\Vert_{0,\beta,p,q}e^{-\beta a^q}\le C_{d,q}\Vert f\Vert_{\alpha,\beta,p,q}e^{-\beta a^q}.
\end{equation}
We aim to bound $\Vert \widetilde{f}\Vert_{K_{\alpha,a,d}}=\Vert f+g\Vert_{K_{\alpha,a,d}}$. Although $f$ and $g$ may not lie in the Korobov space, we can still compute their Korobov norm, and the triangle inequality yields
\begin{equation}
\begin{aligned}
    \Vert \widetilde{f}\Vert_{K_{\alpha,a,d}}\le \Vert f\Vert_{K_{\alpha,a,d}}+\Vert g\Vert_{K_{\alpha,a,d}}\le \Vert f\Vert_\alpha + \Vert g\Vert_{K_{\alpha,a,d}}.\nonumber
\end{aligned}
\end{equation}
Note that for $a\ge\max\{\beta^{-1/q},\frac{1}{2} \}$, we have 
\begin{align}
\Vert g\Vert_{K_{\alpha,a,d}}^2&=\sum_{u\subset\{1{:}d\}}\int_{[-a,a]^{|u|}}\left(\int_{[-a,a]^{d-|u|}}\partial_u^\alpha g(\bm{x})d\bm{x}_{-u}\right)^2d\bm{x}_{u}\nonumber\\
&\le \sum_{u\subset\{1{:}d\}}\int_{[-a,a]^{|u|}}\left(\int_{[-a,a]^{d-|u|}}C_{d,q}\Vert f\Vert_{\alpha,\beta,p,q}e^{-\beta a^q}d\bm{x}_{-u}\right)^2d\bm{x}_{u}\label{g_first}\\
&=C_{d,q}^2\Vert f\Vert_{\alpha,\beta,p,q}^2e^{-2\beta a^q}\sum_{u\subset\{1{:}d\}}(2a)^{2d-|u|}\nonumber\\
&\le  C_{d,q}^2\Vert f\Vert_{\alpha,\beta,p,q}^2 2^d(2a)^{2d}e^{-2\beta a^q}\label{g_second},
\end{align}
where in (\ref{g_first}) we use (\ref{partial_g}) and in (\ref{g_second}) we use $2a\ge 1$. So (\ref{wide_norm}) follows from the inequality $(x+y)^2\le 2x^2+2y^2$.
\end{proof}

To analyze the third term in (\ref{MISE_Rd}), we recall \cite[Proposition 8]{nuyens2023scaled} as follows.

\begin{proposition}\label{truncation}
    Suppose a function $h$ satisfies the  exponential decay condition up to order $0$, i.e., $\Vert h\Vert_{0,\beta,p,q}=\sup_{\bm{x}\in\mathbb{R}^d}|\exp(\beta\Vert\bm{x}\Vert_{p}^q )h(\bm{x})|<\infty$,
    for some $\beta>0,1\le p\le\infty$ and $1\le q<\infty$. Then for any $a \ge \beta^{-1/q}$ we have 
    \begin{equation}
        \left|\int_{\mathbb{R}^d\setminus[-a,a]^d}h(\bm{x})d\bm{x}\right|\le \frac{2^d d}{\beta^{d/q}q}\left\lceil d/q \right\rceil !   \Vert h\Vert_{0,\beta,p,q}(\beta a^q)^{d/q-1}\exp(-\beta a^q).\nonumber
    \end{equation}
\end{proposition}

Combining the results of Theorem~\ref{MISE_Kor}, Corollary~\ref{term1}, Corollary~\ref{term2}, along with Proposition~\ref{truncation}, we arrive at the following theorem.

\begin{theorem}\label{MISE_Rd_sim}
Given $d,\alpha\in\mathbb{N}$. Assume that the target density $f\in S^{\alpha}_{\beta,p,q}$ for some $\beta>0,1\le p\le\infty$ and $1\le q<\infty$. For $\epsilon \in (0,2-1/\alpha)$, take $\tau^*=1/(2\alpha)+\epsilon/2$, $\lambda^* = \mathcal{O}(M^{-1/(1+\tau^*)})$, and prime number $N^*=\mathcal{O}(M^{1/(\alpha-\epsilon)})$. For any $\eta>0$, let $a^* = \left[\frac{1}{2\beta}\log (\eta M)\right]^{1/q}$. Denote $\Vert f\Vert_{\max} := \max\{\Vert f\Vert_{\alpha,\beta,p,q},\Vert f\Vert_{\alpha}\}$ and let $V_{N^*}=V_{N^*}(X^*)$ with $X^*$ defined in (\ref{X*}). Then there exists a constant $C$ depending only on $\alpha,\beta,\eta,d,q$ and $\epsilon$ such that for $M\ge\max\{e^2/\eta,e^{2^{1-q}\beta}/\eta \}$ and the PSKK estimator $f^\lambda_{\bm{Y}}$ in (\ref{estimation_R^d}), we have
    \begin{equation}\label{MISE_order}
        \mathbb{E}\left[\int_{\mathbb{R}^d}\left|f_{\bm{Y}}^\lambda(\bm{x})-f(\bm{x})\right|^2d\bm{x}\right]
        \le C\Vert f\Vert_{\max}^2|\log M|^{\frac{2(\alpha+1)d}{q}}M^{-\frac{1}{1+1/(2\alpha)+\epsilon/2}}.
    \end{equation} 
\end{theorem}

\begin{proof}
Since $\Vert f\Vert_{0,\beta,p,q}\le \Vert f\Vert_{\alpha,\beta,p,q}<\infty$, we have $f\in L^2(\mathbb{R}^d)$ and
\begin{equation}
    \Vert f^2\Vert_{0,2\beta,p,q}=\sup_{\bm{x}\in\mathbb{R}^d}|\exp(2\beta\Vert\bm{x}\Vert_{p}^q )|f(\bm{x})|^2|=\Vert f\Vert_{0,\beta,p,q}^2<\infty.\nonumber
\end{equation}
According to Proposition~\ref{truncation}, for $a\ge(2\beta)^{-1/q}$, we have $2\beta a^q\ge 1$ and 
\begin{align}
\int_{\mathbb{R}^d\setminus[-a,a]^d}|f(\bm{x})|^2d\bm{x}
\le& \frac{2^d d}{(2\beta)^{d/q}q}\left\lceil d/q \right\rceil !   \Vert f\Vert_{0,\beta,p,q}^2(2\beta a^q)^{d/q-1}\exp(-2\beta a^q)\nonumber\\
\le&C_{d,q}'\Vert f\Vert_{0,\beta,p,q}^2(2a)^d\exp(-2\beta a^q),\label{term3}    
\end{align}
where $C_{d,q}'=\frac{d}{q}\left\lceil d/q \right\rceil !$. Substituting inequality~(\ref{term3}), the results of Theorem~\ref{MISE_Kor}, Corollary~\ref{term1}, and Corollary~\ref{term2}, into inequality~(\ref{MISE_Rd}), we obtain that for any $\delta\in (0,\alpha/2),\tau\in(1/(2\alpha),1],\lambda\in (0,1)$ and $a\ge\max\{\beta^{-\frac1q},\frac{1}{2}\}$,
\begin{align}
&\mathbb{E}\left[\int_{\mathbb{R}^d}\left|f_{\bm{Y}}^\lambda(\bm{x})-f(\bm{x})\right|^2d\bm{x}\right]\nonumber\\
\le&2C_{d,q}^2\Vert f\Vert_{0,\beta,p,q}^2(2a)^d e^{-2\beta a^q}+\frac{2\left((2a)^{-\tau-1}+\frac{1}{a}\left(\frac{a}{\pi}\right)^{2\alpha\tau}\zeta(2\alpha\tau)\right)^{d}}{M\lambda^\tau}\nonumber\\
&+4\lambda\left(\Vert f\Vert_\alpha^2+ C_{d,q}^2\Vert f\Vert_{\alpha,\beta,p,q}^2(2\sqrt{2}a)^{2d}e^{-2\beta a^q}\right)\nonumber\\
&+4C^2_{\alpha,\delta,d}\frac{(2a)^{2\alpha d}}{N^{\alpha-2\delta}}\left(\Vert f\Vert_\alpha^2+ C_{d,q}^2\Vert f\Vert_{\alpha,\beta,p,q}^2(2\sqrt{2}a)^{2d}e^{-2\beta a^q}\right)\nonumber\\
&+C_{d,q}'\Vert f\Vert_{0,\beta,p,q}^2(2a)^d\exp(-2\beta a^q),\label{MISE_with_par}    
\end{align}
where $C_{\alpha,\delta,q}$ and $C_{d,q}$ are constants defined in (\ref{constant_C}) and (\ref{C_dq}), respectively. Note that the condition $M\ge\max\{e^2/\eta,e^{2^{1-q}\beta}/\eta \}$ ensures $a^* \ge\max\{\beta^{-1/q},\frac{1}{2} \}$.
By substituting the value of $a^*$ into inequality (\ref{MISE_with_par}), we obtain that there exists a constant $C_1=C_1(\alpha,\beta,\delta,\tau,d,q)$ depending only on $\alpha,\beta,\delta,\tau,d$ and $q$ such that 
\begin{align}
&\mathbb{E}\left[\int_{\mathbb{R}^d}\left|f_{\bm{Y}}^\lambda(\bm{x})-f(\bm{x})\right|^2d\bm{x}\right]\nonumber\\
\le& C_1\Vert f\Vert_{\max}^2 
\left[\frac{|\log(\eta M)|^{\frac{d}{q}}}{\eta M}+\frac{|\log(\eta M)|^{\frac{2\alpha d}{q}}}{N^{\alpha-2\delta}}\left(1+\frac{|\log(\eta M)|^{\frac{2d}{q}}}{\eta M}\right)\right.\nonumber\\
&\left.+\lambda\left(1+\frac{|\log(\eta M)|^{\frac{2d}{q}}}{\eta M}\right) + \frac{|\log(\eta M)|^{\frac{(2\alpha\tau -1)d}{q}}}{M\lambda^\tau}\right].    \label{MISE_without_a}
\end{align}
Taking $\delta=\epsilon/2$ and plugging the values of $\tau^*,\lambda^*$ and $N^*$ into inequality~(\ref{MISE_without_a}), we obtain that there exist  constants $C_2 = C_2(\alpha,\beta,\eta,d,q,\epsilon)$ and $C=C(\alpha,\beta,\eta,d,q,\epsilon)$ such that
\begin{align}
&\mathbb{E}\left[\int_{\mathbb{R}^d}\left|f_{\bm{Y}}^\lambda(\bm{x})-f(\bm{x})\right|^2d\bm{x}\right]\nonumber\\
\le& C_2\Vert f\Vert_{\max}^2 \left[\frac{(\log M)^{\frac{d}{q}}}{M}+\frac{(\log M)^{\frac{2\alpha d}{q}}}{M}\left(1+\frac{(\log M)^{\frac{2d}{q}}}{M}\right)\right.\nonumber\\
&\left.+M^{-\frac{1}{1+\tau^*}}\left(1+\frac{(\log M)^{\frac{2d}{q}}}{M}+(\log M)^{\frac{\alpha\epsilon d}{q}}\right)\right]\nonumber\\
\le& C\Vert f\Vert_{\max}^2 M^{-\frac{1}{1+\tau^*}}|\log M|^{\frac{2(\alpha+1)d}{q}}\nonumber\\
=&  C\Vert f\Vert_{\max}^2|\log M|^{\frac{2(\alpha+1)d}{q}}M^{-\frac{1}{1+1/(2\alpha)+\epsilon/2}}.\nonumber
\end{align}
\end{proof}

\begin{remark}
    Theorem~\ref{MISE_Rd_sim} shows that our PSKK estimator achieves an MISE convergence rate of $\mathcal{O}(|\log M|^{2(\alpha+1)d/q}M^{-1/(1+1/(2\alpha)+\epsilon/2)})$
    for arbitrarily small $\epsilon > 0$ under suitable parameter choices.
    Asymptotically, this yields a convergence rate of $\mathcal{O}(M^{-1/(1+1/(2\alpha)+\epsilon)})$, which matches the convergence rate obtained in \cite{kazashi2023density}
    while extending the applicability to $\mathbb{R}^d$ density estimation.  However, a limitation stems from the factor $|\log M|^{2(\alpha+1)d/q}$, which grows rapidly with the dimension $d$. Thus the PSKK method may face challenges in high-dimensional settings.
\end{remark}

\begin{remark}
    In related works on numerical integration \cite{kuo2003component,sloan1998quasi,sloan2001tractability} and density estimation \cite{kazashi2023density}, weighted function spaces are used to achieve dimension-independent convergence by exploiting the favorable anisotropic structure of the target function. Similarly, 
    if we were to consider a weighted scaled Korobov space---that is, replacing $r_{\alpha,a,d}(\bm{h})$ in the expression for $K_{\alpha,a,d}$ in (\ref{kernel_ser}) with 
    \begin{equation*}
        r_{\alpha,a,d,\bm{\gamma}}(\bm{h}):=\gamma_{s(\bm{h})}\prod_{j=1}^{d}r_{\alpha,a}(h_j),\quad s(\bm{h}):=\{j\in \{1{:}d\}:h_j\ne 0 \},
    \end{equation*}
    where $\bm{\gamma}=(\gamma_u)_{u\subset\{1{:}d\} }$ is the set of weights and $r_{\alpha,a}$ is defined in (\ref{r})---we may achieve a dimension-independent convergence result only for density estimation over $\mathcal{H}(K_{\alpha,a,d})$ in Theorem \ref{MISE_Kor}, which may improve the convergence rate of the second term in (\ref{MISE_Rd}). However, a residual dependence on $d$ would persist due to the density truncation error (see Proposition~\ref{truncation}) and the discrepancy between the wrapped density $\widetilde{f}$ and the original density $f$ (see Corollary~\ref{term1} and Corollary~\ref{term2}). These components introduce a factor of the form $|\log M|^{2d/q}$ into the MISE bound, which would lead to a total convergence rate of $\mathcal{O}(|\log M|^{2 d/q}M^{-1/(1+1/(2\alpha)+\epsilon)})$ under an approach similar to Theorem~\ref{MISE_Rd_sim}. This reveals that the dimensional dependence is intrinsic to the method and cannot be eliminated by working with weighted spaces. 
    For this reason, we use the unweighted space in this work. This choice avoids the necessity of selecting appropriate weights, which is often challenging due to a lack of prior information about the target density. Moreover, it circumvents the strong anisotropic assumptions required for dimension-independent convergence, which may be unjustified in practical problems where the density function $f$ exhibits comparable dependence on different groups of variables. This approach also yields computational benefits, as the constructed scaled lattice point set $X^*$ becomes reusable across different problems.

\end{remark}

\section{Implementing the PSKK method}\label{sec5}
We outline our PSKK method as Algorithm~\ref{alg:buildtree}. First,  set the theoretical parameters mentioned in Theorem~\ref{MISE_Rd_sim}. Second, preprocess the data by applying modulo reduction into $[-a,a)^d$ and use the CBC algorithm in \cite{cools2021fast} to generate lattice points. Third, numerically solve the linear system~(\ref{linear system}). Notably, when using lattice points to construct $X^*$ in (\ref{X*}), the matrix $A$ in (\ref{linear system}) exhibits a circulant matrix structure that enables efficient computation. The detailed
calculation of $A$ is provided in Appendix~\ref{SM_computing_A}. Finally, we obtain the PSKK estimator as in (\ref{estimation_R^d}).

\begin{algorithm}
\caption{Periodic scaled Korobov kernel method}
\label{alg:buildtree}
\begin{algorithmic}[1]
\REQUIRE{$M$ observed samples $Y_1,\ldots,Y_M\in\mathbb{R}^d$. Prior knowledge that $f\in S^\alpha_{\beta,p,q}$ for some $\alpha\in\mathbb{N},\beta>0,1\le p\le\infty$ and $1\le q<\infty$.}
\STATE{Select $\epsilon \in (0,2-1/\alpha)$ and $\eta>0$. Take a prime number $N$ such that $N=\mathcal{O}(M^{\frac{1}{\alpha-\epsilon}})$. Take
    \begin{equation*}
        a =  \left(\frac{\log M+\log\eta}{2\beta}\right)^{\frac{1}{q}},\quad \bm{a}=(a,\ldots,a)^\top\in\mathbb{R}^d,\quad \lambda =\mathcal{O}(M^{-1/(1+\frac{1}{2\alpha}+\frac{\epsilon}{2})}).
    \end{equation*}}
\STATE{For $1\le m\le M$, let $\widetilde{Y}_m \in [-a,a)^d$ be the remainder of $Y_m$ modulo $2a$ minus $\bm{a}$. }
\STATE{Use the CBC construction in \cite{cools2021fast} to obtain the lattice point set $Y^* = \{\bm{y}_1,\ldots,\bm{y}_N\}\subset [0,1]^d$. Take $\bm{x}_i:=2a\bm{y}_i-\bm{a},i=1,\cdots,N$.}
\STATE{Employ the density estimation in $\mathcal{H}(K_{\alpha,a,d})$ as follows:}
\STATE {\hspace*{2em} Take $A=(A_{j,k})_{1\le j,k\le N}$ and $\bm{b}=(b_j)_{1\le j\le N}$ such that 
\begin{align*}
    A_{j,k}=&\left<K_{\alpha,a,d}(\bm{x}_j,\cdot),K_{\alpha,a,d}(\bm{x}_k,\cdot)\right>_{L^2([-a,a]^d)}+ \lambda K_{\alpha,a,d}(\bm{x}_j,\bm{x}_k),
\end{align*}
\begin{align}
    b_j = \frac{1}{M}\sum_{m=1}^{M}K_{\alpha,a,d}(\bm{x}_j,\widetilde{Y}_{m}).\label{vector_b}
\end{align}}
\STATE{\hspace*{2em} Solve the linear equation $A\bm{c}=\bm{b}$.

\STATE{\hspace*{2em} For the solution $\bm{c}=(c_1,\ldots,c_N)^\top$, compute the density estimator
\begin{equation*}
    \widetilde{f}^\lambda_{\widetilde{\bm{Y}}}(\bm{x})=\sum_{n=1}^{N}c_nK_{\alpha,a,d}(\bm{x}_n,\bm{x}).
\end{equation*}}

\RETURN Density estimator 
\[
    f^{\lambda}_{\bm{Y}}(\bm{x}):=\begin{cases}\max\{\widetilde{f}^\lambda_{\widetilde{\bm{Y}}}(\bm{x}),0 \},
     &\text{ for }\bm{x}\in[-a,a]^d,\\
    0, &\text{ for }\bm{x}\in\mathbb{R}^d\setminus[-a,a]^d.
    \end{cases}\]
}
\end{algorithmic}
\end{algorithm}

\subsection{Computational complexity of the PSKK method}
In this subsection we analyze the overall computational complexity of the PSKK method. First, the CBC construction in \cite{cools2021fast} for generating the lattice point set $Y^*$ requires $\mathcal{O}(dN\log N)$ operations under product weights in the Korobov space on the unit cube (see \cite[Section 4]{cools2021fast} for detail). Next, constructing the circulant matrix $A$ involves computing its first row elements $A_{1,k}$ for $k=1,2,\ldots N$ through kernel evaluations in (\ref{kernel_form}) and (\ref{L2_form}), which takes $\mathcal{O}(dN)$ operations. For the vector $\bm{b}$, each entry requires $\mathcal{O}(dM)$ operations according to (\ref{vector_b}), leading to $\mathcal{O}(dMN)$ complexity for all $N$ entries. Finally, solving the linear system $A\bm{c}=\bm{b}$ via fast Fourier transform leverages the circulant structure of $A$, requiring $\mathcal{O}(N\log N)$ operations. Therefore, the total computational complexity of the PSKK method is dominated by 
\begin{equation}
    \mathcal{O}(dN\log N+dN+dMN+N\log N) = \mathcal{O}(dMN).\nonumber
\end{equation}
By selecting $N=\mathcal{O}(M^{\frac{1}{\alpha-\epsilon}})$ for some $\epsilon \in (0,2-1/\alpha)$ as in Algorithm~\ref{alg:buildtree}, the complexity reduces to $\mathcal{O}(dM^{1+\frac{1}{\alpha-\epsilon}})$.

\begin{remark}
    The PSKK estimator is a function defined on $\mathbb{R}^d$. Evaluating this estimator at $L$ distinct points requires $\mathcal{O}(dLN)=\mathcal{O}(dLM^{\frac{1}{\alpha-\epsilon}})$ operations after construction.
\end{remark}

\begin{remark}
    For comparison, the standard KDE method with Scott’s rule takes $\mathcal{O}(dLM)$ operations to evaluate $L$ points. The PSKK method achieves a lower complexity than KDE when $L>cM^\rho$ with some positive constant $c$ and $\rho>\frac{1}{\alpha-\epsilon}$, provided $N=\mathcal{O}(M^{\frac{1}{\alpha-\epsilon}})$ as specified in Algorithm~\ref{alg:buildtree}. This condition ensures $\mathcal{O}(dLN+dMN)< \mathcal{O}(dLM)$, demonstrating the computational efficiency of our method for large $L$. Combined with this low computational complexity, the PSKK method achieves a convergence rate of $\mathcal{O}(M^{-1/(1+1/(2\alpha)+\epsilon)})$, exceeding the $\mathcal{O}(M^{-\frac{4}{d+4}})$ rate of the standard KDE method.
\end{remark}

\section{Numerical experiments}\label{sec6}
In this section we perform numerical experiments to validate the theoretical results. All experiments compare the PSKK method with the standard KDE method, and evaluate their performance through the MISE. We provide detailed steps for calculating MISE in Appendix~\ref{SM_MISE_steps}.

\subsection{Example 1: Two-dimensional Gaussian mixture distribution}

In this subsection, we examine our PSKK method on the two dimensional Gaussian mixture distribution. In particular, we consider the following density function       
\begin{equation}
    f(\bm{x})=\frac{1}{9}\sum_{\bm{\mu}\in\{-2,0,2\}^2}\frac{2}{\pi}e^{-2\Vert \bm{x}-\bm{\mu}\Vert_2^2}.\nonumber
\end{equation}
This density function has nine peaks, with the center of each peak located at the grid points in $\{-2,0,2\}^2$. Now we compare the performance of the KDE method and the PSKK method for this density estimation. For the KDE method, we select the conventional Gaussian kernel. For the PSKK method, we chose $\alpha=2,a=6,N=1009$ and $\lambda=10^{-6}$.

\begin{figure}[htbp]
  \centering
  \includegraphics[width=\linewidth]{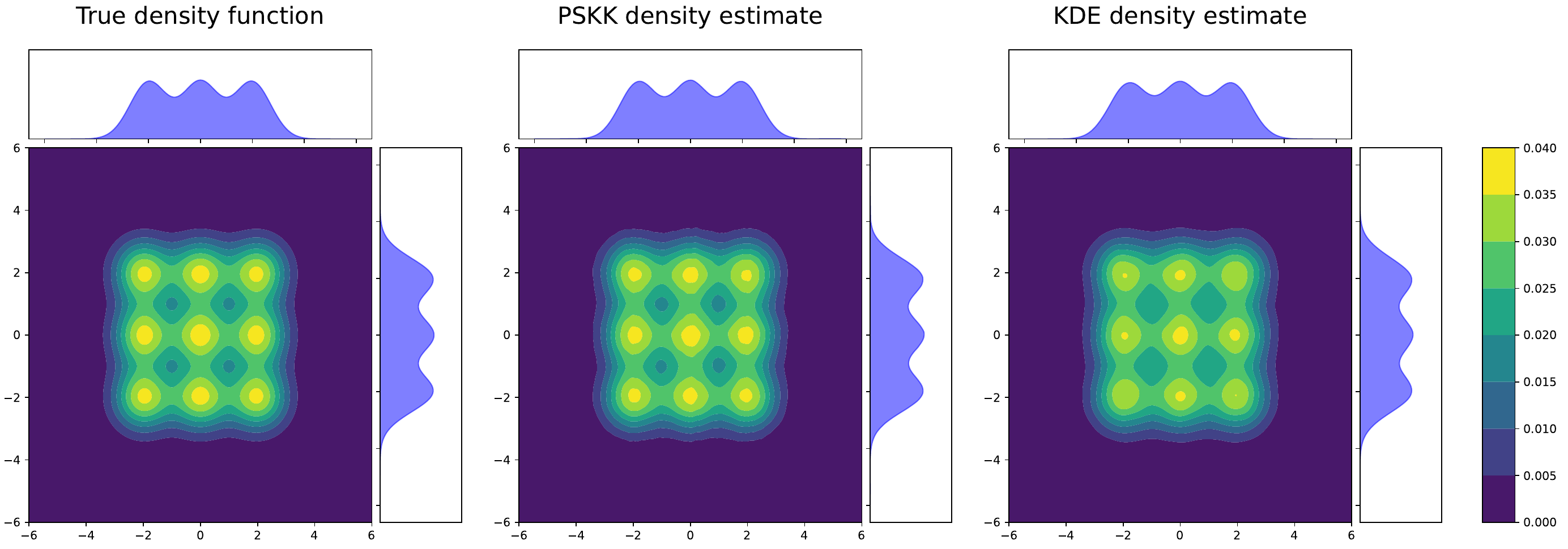}
  \caption{Plot of the density function with marginal distributions for the true density function (left),  the  PSKK density estimate (middle) and the KDE density estimate (right), based on $M=10^6$ samples. The PSKK method sets the parameters $\alpha=2,a=6,N=1009$ and $\lambda=10^{-6}$.}
  \label{fig1:2dim_plt} 
\end{figure}

Figure~\ref{fig1:2dim_plt} shows the true density function (left), the density function estimated by the PSKK method (middle), and the density function estimated by the KDE method (right). Each method employs $M=10^6$ independent samples. It is evident that while both methods perform well in estimating the marginal density functions, the PSKK method more effectively captures the modal information of the target density function, particularly in the peaks (yellow regions) and valleys (dark blue regions).

\begin{figure}[htbp]
  \centering
  \includegraphics[width=0.7\linewidth]{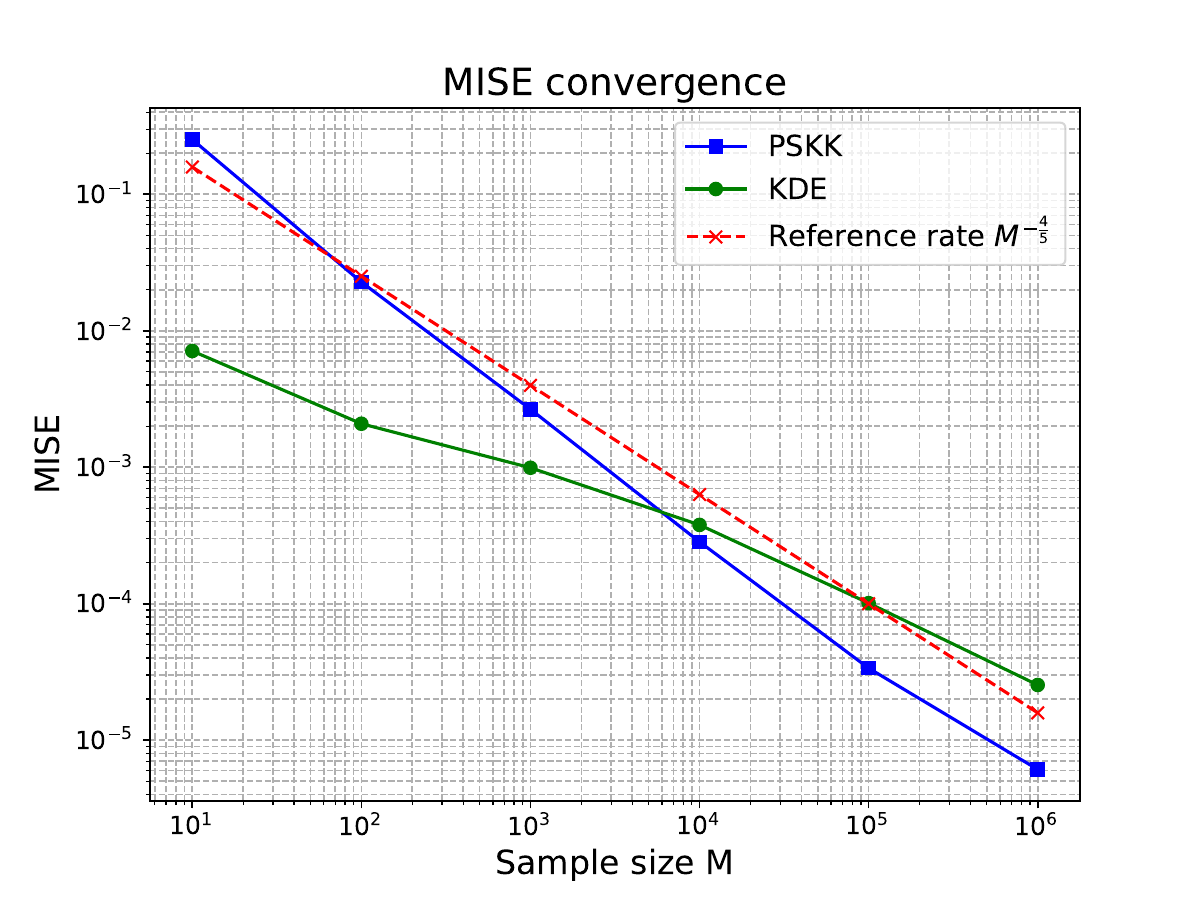}
  \caption{MISE convergence for the KDE method and the PSKK method (with $\alpha=2, a=6, N=1009$, and $ \lambda=10^{-6}$) in the $2$-dimensional Gaussian mixture example. }
  \label{fig2:2dim_MISE} 
\end{figure}

We also evaluate the convergence of the MISE for both the KDE and the PSKK methods using sample sizes of $M=10,10^2,\ldots,10^6$.  As illustrated in Figure~\ref{fig2:2dim_MISE}, the PSKK method, under the chosen parameters, exhibits a lower MISE than the KDE method when the sample size exceeds $10^4$, thereby highlighting its superior performance with large samples. Furthermore, when $\alpha=2$, the PSKK method in this example exhibits a higher MISE convergence rate than the theoretical upper bound of $\mathcal{O}(M^{-\frac{4}{5}})$.

\subsection{Example 2: 4-dimensional Gaussian mixture distribution}\label{Example2}
In this subsection we consider the 4-dimensional Gaussian mixture distribution with the density function
\begin{equation*}
    f(\bm{x})=\frac{1}{9}\sum_{k=0}^{8}\frac{1}{4\pi^2\sigma^4}e^{-\frac{\Vert \bm{x}-\bm{\mu}_k\Vert_2^2}{2\sigma^2}},
\end{equation*}
where $\sigma=0.7$ and for $0\le k\le 8$,
\begin{equation}
    \bm{\mu}_{k}=\left(\left\{\frac{k}{9}\right\}-\frac{4}{9},\left\{\frac{2k}{9}\right\}-\frac{4}{9},\left\{\frac{4k}{9}\right\}-\frac{4}{9},\left\{\frac{8k}{9}\right\}-\frac{4}{9}\right)^\top.\nonumber
\end{equation}
We aim to verify the convergence rate of the MISE of the PSKK estimator established in Theorem~\ref{MISE_Rd_sim}, using the parameter selection strategy specified in Algorithm~\ref{alg:buildtree}.

\begin{figure}[htbp]
  \centering
  \includegraphics[width=0.7\linewidth]{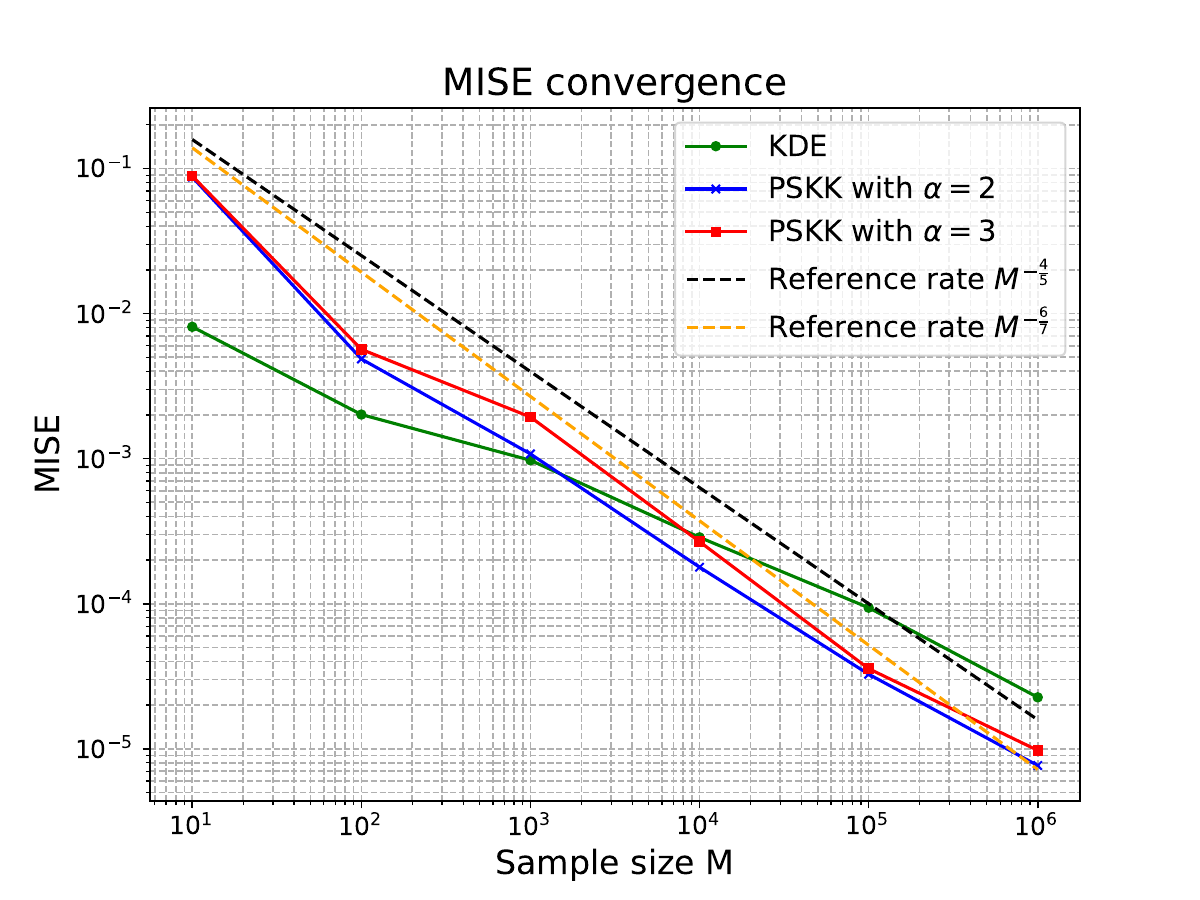}
  \caption{MISE convergence for the KDE method and the PSKK method (with $\alpha=2,3,a=\sqrt{(\log M-1)/2},N=\mathcal{O}(3\sqrt{M}),$ and $\lambda=0.1M^{-1/(1+1/(2\alpha))}$) in the $4$-dimensional Gaussian mixture example. }
  \label{fig6:4dim_fin} 
\end{figure}

Figure~\ref{fig6:4dim_fin} illustrates the MISE decay for the PSKK method with $a=a(M),N=N(M,\alpha)$ and $\lambda=\lambda(M,\alpha)$ for $\alpha=2,3$ as specified in Algorithm~\ref{alg:buildtree}. Specifically, we take $a = \sqrt{(\log M-1)/2}$ and $\lambda=0.1M^{-1/(1+1/(2\alpha))}$ for $M=10,10^2,\ldots,10^6$. For the selection of $N$, we choose $N$ to be approximately $3\sqrt{M}$ prime numbers. For $M=10,10^2,\ldots,10^6$, $N$ is set to $11,31,97,307,947,3001$, respectively. In both cases of $\alpha=2$ and $\alpha=3$, the MISE of the PSKK method exhibits a $\mathcal{O}(M^{-\frac{1}{1+1/(2\alpha)}})$ decay, which is consistent with our theory. Additionally, we evaluated the MISE decay for the KDE method. The results indicate that, as the sample size $M$ increases, the PSKK method yields a smaller MISE compared to the KDE method.

We also investigate the impact of parameters $a, N,$ and $\lambda$ on the MISE of the PSKK estimator and conduct similar experiments for higher-dimensional situations. The detailed results of the parameter impact experiments and the higher-dimensional experiments are provided in Appendix~\ref{SM_ex2_parameters} and Appendix~\ref{SM_higher_examples}, respectively.

\section{Conclusion}\label{sec7}
In this paper, we propose the periodic scaled Korobov kernel method for density estimation on 
$\mathbb{R}^d$ under the exponential decay condition. By employing a modulo operation to map samples into the truncated domain $[-a,a]^d$, the PSKK method eliminates the need for compact support or inherent periodicity assumptions required by the kernel method in \cite{kazashi2023density}. Theoretically, we prove that the PSKK estimator achieves an MISE bound of $\mathcal{O}(|\log M|^{2(\alpha+1)d/q}M^{-1/(1+1/(2\alpha)+\epsilon)})$ for densities with exponential decay up to order $\alpha$, where $q\ge 1$ is a parameter related to the decay of the density function and $\epsilon>0$ is arbitrarily small. Numerical experiments confirm the asymptotic convergence rate of the PSKK method and demonstrate significant improvements over the standard KDE method, especially for large sample sizes.

Some challenges remain to be addressed in future research. First, the MISE bound for our method suffers from a logarithmic factor whose exponent increases with the dimension $d$. This intrinsic dependence on $d$, which cannot be resolved by working with weighted function spaces, may limit the effectiveness of the method in high-dimensional settings.
Second, the number of pre-selected points $N$ is important for our method. Empirically, a value of $N$ approximately $1000$ is typically required to ensure a small error. Therefore, it is necessary to develop a more efficient kernel approximation in the RKHS to resolve the accuracy-complexity trade-off. Third, the scaling parameter $a$ requires precise calibration---both undersized and over-sized $a$ values destabilize the MISE convergence. This highlights the need for creating suitable selection strategies for the scaling parameter.

\section*{Acknowledgments}
The authors sincerely appreciate the anonymous referees for their valuable suggestions and comments, which significantly improved the quality of this paper.

\appendix
\section[Proof of Lemma]{Proof of Lemma~\ref{point_evaluation_functional}}\label{SM_point_evaluation_functional}
\begin{proof}
    For $v\in\mathcal{N}^\tau(K_{\alpha,a,d})$, we have 
    \begin{equation}
    \begin{aligned}
        |\delta_{\bm{x}}(v)|&=\left|\sum_{\bm{h}\in\mathbb{Z}^d}\widehat{v}_{a,d}(\bm{h})\varphi_{a,d,\bm{h}}(\bm{x})\right|\\
        &\le (2a)^{-\frac{d}{2}} \left|\sum_{\bm{h}\in\mathbb{Z}^d}\widehat{v}_{a,d}(\bm{h})\right|   \\
        &\le (2a)^{-\frac{d}{2}}\left( \sum_{\bm{h}\in\mathbb{Z}^d}r_{\alpha,a,d}(\bm{h})^{-2\tau} \right)^{\frac{1}{2}}
        \left(\sum_{\bm{h}\in\mathbb{Z}^d} r_{\alpha,a,d}(\bm{h})^{2\tau}\left|\widehat{v}_{a,d}(\bm{h})\right|^2 \right)^{\frac{1}{2}}\\
        &=(2a)^{-\frac{d}{2}}\left( \sum_{\bm{h}\in\mathbb{Z}^d}r_{\alpha,a,d}(\bm{h})^{-2\tau} \right)^{\frac{1}{2}} \Vert v\Vert_{\mathcal{N}^\tau(K_{\alpha,a,d})},\nonumber
    \end{aligned}
    \end{equation}
    where in the second inequality we use $\varphi_{a,d,\bm{h}}(\bm{x})\le (2a)^{-\frac{d}{2}},\forall \bm{x}\in[-a,a]^d$. Since $\tau>\frac{1}{2\alpha}$, we have
    \begin{equation}
    \begin{aligned}
        \sum_{\bm{h}\in\mathbb{Z}^d}r_{\alpha,a,d}(\bm{h})^{-2\tau}
        &=\left(r_{\alpha,a}(0)^{-2\tau}+2\sum_{h=1}^{\infty}r_{\alpha,a}(h)^{-2\tau} \right)^d\\
        &=\left((2a)^{-\tau}+2\left(\frac{a}{\pi}\right)^{2\alpha\tau}\zeta(2\alpha\tau)\right)^{d}<\infty,\nonumber
    \end{aligned}
    \end{equation}
    where $\zeta(s):=\sum_{n=1}^\infty \frac{1}{n^s}$ for $s>1$ is the Riemann zeta function. Therefore, $\delta_{\bm{x}}\in\mathcal{N}^{-\tau}(K_{\alpha,a,d})$ holds for any $\bm{x}\in[-a,a]^d$.
\end{proof}

\section[Proof of Lemma]{Proof of Lemma~\ref{F}}\label{SM_F}
\begin{proof}
    \begin{equation}
    \begin{aligned}
        \Vert F\Vert_{\mathcal{N}^{-\tau}(K_{\alpha,a,d})}^2
        &=\sum_{\bm{h}\in\mathbb{Z}^d}r_{\alpha,a,d}(\bm{h})^{-2\tau}\left|\int_{[-a,a]^d}f(\bm{x}) \varphi_{a,d,\bm{h}}(\bm{x})d\bm{x}\right|^2\\
        &\le\sum_{\bm{h}\in\mathbb{Z}^d}r_{\alpha,a,d}(\bm{h})^{-2\tau}(2a)^{-d}\left|\int_{[-a,a]^d}f(\bm{x})d\bm{x}\right|^2\\
        &=(2a)^{-d}\sum_{\bm{h}\in\mathbb{Z}^d}r_{\alpha,a,d}(\bm{h})^{-2\tau}\\
        &=\left((2a)^{-\tau-1}+\frac{1}{a}\left(\frac{a}{\pi}\right)^{2\alpha\tau}\zeta(2\alpha\tau)\right)^{d}<\infty,\nonumber
    \end{aligned}
    \end{equation}
    where in the third equality we use the fact that $f$ is a density function on $[-a,a]^d$.
\end{proof}

\section[Proof of Lemma]{Proof of Lemma~\ref{L2_K_tau}}\label{SM_L2_K_tau}
\begin{proof}
    \begin{align}
        \left\langle K_\tau(\cdot,\cdot),f(\cdot)\right\rangle_{L^2([-\bm{a},\bm{a}])}
        &=(2a)^{-d}\sum_{\bm{h}\in\mathbb{Z}^d}r_{\alpha,a,d}(\bm{h})^{-2\tau}\int_{[-a,a]^d}f(\bm{x})d\bm{x}\nonumber\\
        &=(2a)^{-d}\sum_{\bm{h}\in\mathbb{Z}^d}r_{\alpha,a,d}(\bm{h})^{-2\tau}\nonumber\\
        &=\left((2a)^{-\tau-1}+\frac{1}{a}\left(\frac{a}{\pi}\right)^{2\alpha\tau}\zeta(2\alpha\tau)\right)^{d}<\infty.\nonumber    
    \end{align}
\end{proof}

\section[Proof of Prop]{Proof of Proposition~\ref{L_2_A_N}}\label{SM_proof_L_2_A_N}
\begin{proof}
    For any $\phi\in\mathcal{H}(K_{\alpha,a,d})$, let
    $h(\bm{y}) = \phi\circ\mathcal{J}_a(\bm{y})=\phi(2a\bm{y}-\bm{a})$, where $\bm{a}=(a,\ldots,a)^\top$. Let $h_N = B_N^*(h)$ and let $\phi_N = A_N^*(\phi)$.
    Then $\phi = h\circ\mathcal{J}_a^{-1}, \phi_N = h_N\circ\mathcal{J}_a^{-1}$, and we have
    \begin{equation}
    \begin{aligned}
        \Vert\phi-A_N^*(\phi)\Vert_{L^2([-a,a]^d)}^2
        =&\int_{[-a,a]^d}|\phi(\bm{x})-\phi_N(\bm{x})|^2d\bm{x}\\
        =&\int_{[-a,a]^d}\left|h\circ\mathcal{J}_a^{-1}(\bm{x})-h_N\circ\mathcal{J}_a^{-1}(\bm{x})\right|^2d\bm{x}\\
        =&(2a)^d\int_{[0,1]^d}|h(\bm{y})-h_N(\bm{y})|^2d\bm{y},\nonumber
    \end{aligned}
    \end{equation}
    According to Lemma~\ref{L2_B_N} and inequality~(\ref{norm_h}), we have
    \begin{align*}
        \Vert\phi-A_N^*(\phi)\Vert_{L^2([-a,a]^d)}
        =&(2a)^{d/2}\Vert h-B_N^*(h)\Vert_{L^2([0,1]^d)}\\
        \le&(2a)^{d/2}C_{\alpha,\delta,d}\Vert h\Vert_{K_{\alpha,d}}N^{-(\alpha/2-\delta)}\\
        \le&C_{\alpha,\delta,d}\Vert \phi\Vert_{K_{\alpha,a,d}}(2a)^{\alpha d} N^{-(\alpha/2-\delta)}.
    \end{align*}
    
\end{proof}

\section{Computing the elements of $A$}\label{SM_computing_A}
Here we introduce how to calculate $A$ mentioned in Algorithm~\ref{alg:buildtree} of the form $A=(A_{j,k})_{1\le j,k\le N}$ with 
\begin{align}
    A_{j,k}=&\left<K_{\alpha,a,d}(\bm{x}_j,\cdot),K_{\alpha,a,d}(\bm{x}_k,\cdot)\right>_{L^2([-a,a]^d)}+ \lambda K_{\alpha,a,d}(\bm{x}_j,\bm{x}_k).\label{A_jk}
\end{align}

We note that for any $k\in\mathbb{N},k\ge 2$, the periodic Bernoulli polynomial of degree $k$ has the following form (see \cite{NIST_DLMF_2019} for more details)
\begin{equation}
    \frac{\widetilde{B}_{k}(x)}{k!}=\frac{-1}{(2\pi i)^{k}}\sum_{h\in\mathbb{Z}\setminus\{0\}}\frac{e^{2\pi i h x}}{h^{k}},\quad\forall x\in\mathbb{R}.\nonumber
\end{equation}
And for any $k\in\mathbb{N}$, $x\in [0,1]$, $\widetilde{B}_{2k}(x) = B_{2k}(|x|)$.
So we have
\begin{align}
K_{\alpha,a,d}(\bm{x},\bm{y})=&\prod_{j=1}^{d}\left(\frac{1}{(2a)^2}+\frac{(-1)^{\alpha+1}(2a)^{2\alpha-1}}{(2\alpha)!}B_{2\alpha}\left(\frac{|x_j-y_j|}{2a}\right) \right)\nonumber\\
   =&\prod_{j=1}^{d}\left(\frac{1}{(2a)^2}+\sum_{h\in\mathbb{Z}\setminus\{0\}}\frac{(2a)^{2\alpha-1}}{|2\pi h|^{2\alpha}} \exp\left(2\pi ih\frac{x_j-y_j}{2a}\right)\right).\label{kernel_form}    
\end{align}
Thus
\begin{align}
&\left\langle K_{\alpha,a,d}(\bm{x},\cdot),K_{\alpha,a,d}(\bm{y},\cdot)\right\rangle_{L^2([-\bm{a},\bm{a}])}\nonumber\\
   =&\prod_{j=1}^{d}\left[\int_{-a}^{a}\left(\frac{1}{(2a)^2}+\sum_{h\in\mathbb{Z}\setminus\{0\}}\frac{(2a)^{2\alpha-1}}{|2\pi h|^{2\alpha}} \exp\left(2\pi ih\frac{x_j-z_j}{2a}\right)\right)\right.\nonumber\\
   &\left.\quad \quad\times\left(\frac{1}{(2a)^2}+\sum_{h\in\mathbb{Z}\setminus\{0\}}\frac{(2a)^{2\alpha-1}}{|2\pi h|^{2\alpha}} \exp\left(2\pi ih\frac{z_j-y_j}{2a}\right)\right)dz_j\right]\nonumber\\
   =&\prod_{j=1}^{d}\left((2a)^{-3}+\sum_{h\in\mathbb{Z}\setminus\{0\}}\frac{(2a)^{4\alpha-1}}{|2\pi h|^{4\alpha}}\exp\left( 2\pi i h\frac{x_j-y_j}{2a}\right) \right)\nonumber\\
   =&\prod_{j=1}^{d}\left((2a)^{-3}-\frac{(2a)^{4\alpha-1}}{(4\alpha)!}B_{4\alpha}\left(\frac{|x_j-y_j|}{2a}\right) \right).\label{L2_form}
\end{align}

\begin{remark}\label{remark_A}
    According to (\ref{kernel_form}), (\ref{L2_form}) and (\ref{A_jk}), when the point set $X^*$ is taken as scaled lattice points, the matrix $A$ becomes a circular matrix, and thus equation (\ref{linear system}) can be solved quickly using the fast Fourier transform \cite{duhamel1990fast}.
\end{remark}

\section{Detailed steps for calculating MISE}\label{SM_MISE_steps}
For the KDE estimator $f^\mathrm{KDE}_{\bm{Y}}$ based on the sample set $\bm{Y}=\{Y_1,\ldots,Y_M\}$, the MISE is estimated as 
\begin{align}
\mathrm{MISE}\left(f^\mathrm{KDE}_{\bm{Y}}\right):&=\mathbb{E}\left[\int_{\mathbb{R}^d}|f^\mathrm{KDE}_{\bm{Y}}(\bm{x})-f(\bm{x})|^2d\bm{x}\right]\nonumber\\
&\approx \mathbb{E}\left[\int_{[-l,l]^d}|f^\mathrm{KDE}_{\bm{Y}}(\bm{x})-f(\bm{x})|^2d\bm{x}\right]\nonumber\\
&\approx \frac{(2l)^d}{S}\sum_{s=1}^{S}\frac{1}{2^t}\sum_{k=1}^{2^t}|f^\mathrm{KDE}_{\bm{Y}^{(s)}}(\bm{p}^l_k)-f(\bm{p}^l_k)|^2, \label{KDE_mise}     
\end{align}
where $l>0$ is a truncation parameter, $\bm{Y}^{(1)},\ldots,\bm{Y}^{(S)}$ are Monte Carlo replications of the sample set and 
\begin{equation}
    \bm{p}^l_k=2l\bm{w}_k - l\mathbf{1}\in[-l,l]^d,\quad k=1,2,\ldots,2^t,\nonumber
\end{equation}
with $\mathbf{1}$ being the $d$-dimensional all-ones vector and $\{\bm{w}_k\}_{k=1}^{2^t}$ being the first $2^t$ points of the $d$-dimensional Sobol' sequence, which is mentioned in \cite{dick2013high}.

For the PSKK estimator $f^{\lambda}_{\bm{Y}}$,  we decompose the MISE as
\begin{equation}
\begin{aligned}
\mathrm{MISE}\left(f^{\lambda}_{\bm{Y}}\right):&=\mathbb{E}\left[\int_{\mathbb{R}^d}|f^{\lambda}_{\bm{Y}}(\bm{x})-f(\bm{x})|^2d\bm{x}\right]\\
&=\mathbb{E}\left[\int_{[-a,a]^d}|f^{\lambda}_{\bm{Y}}(\bm{x})-f(\bm{x})|^2d\bm{x}\right]+\int_{\mathbb{R}^d\setminus[-a,a]^d}f(\bm{x})^2d\bm{x}.  \nonumber
\end{aligned}
\end{equation}
The first term on the right hand side is estimated via
\begin{equation}
\mathbb{E}\left[\int_{[-a,a]^d}|f^{\lambda}_{\bm{Y}}(\bm{x})-f(\bm{x})|^2d\bm{x}\right]\approx 
\frac{(2a)^d}{S}\sum_{s=1}^{S}\frac{1}{2^t}\sum_{k=1}^{2^t}|f^{\lambda}_{\bm{Y}^{(s)}}(\bm{p}^a_k)-f(\bm{p}^a_k)|^2,
\end{equation}
where $\bm{Y}^{(s)}$ for $s=1,\ldots,S$ and $\bm{p}^a_k$ for $k=1,\ldots,2^t$ follow the same definitions as in (\ref{KDE_mise}). For the second term, we compute it using $10^8$ Monte Carlo samples drawn from $f(\bm{x})$.

In our numerical experiments, we set $l=6,t=16$ and $S=20$. The truncation parameter $l$ for the MISE of KDE ensures that the MISE truncation error is at least one order of magnitude smaller than the estimated MISE of KDE, while $S$ guarantees that the Monte Carlo confidence interval width is at least ten times smaller than the MISE value.

\section{Supplement to Example 2: Impact of parameters $a, N,$ and $\lambda$ on MISE}\label{SM_ex2_parameters}
This supplement extends the 4-dimensional Gaussian mixture density estimation study from Example 2 (Subsection~\ref{Example2}). Using the same target density, we systematically examine how parameters $\alpha$, $a$, $N$, and $\lambda$ influence the MISE of the PSKK estimator.

\begin{figure}[htbp]
  \centering
  \includegraphics[width=\linewidth]{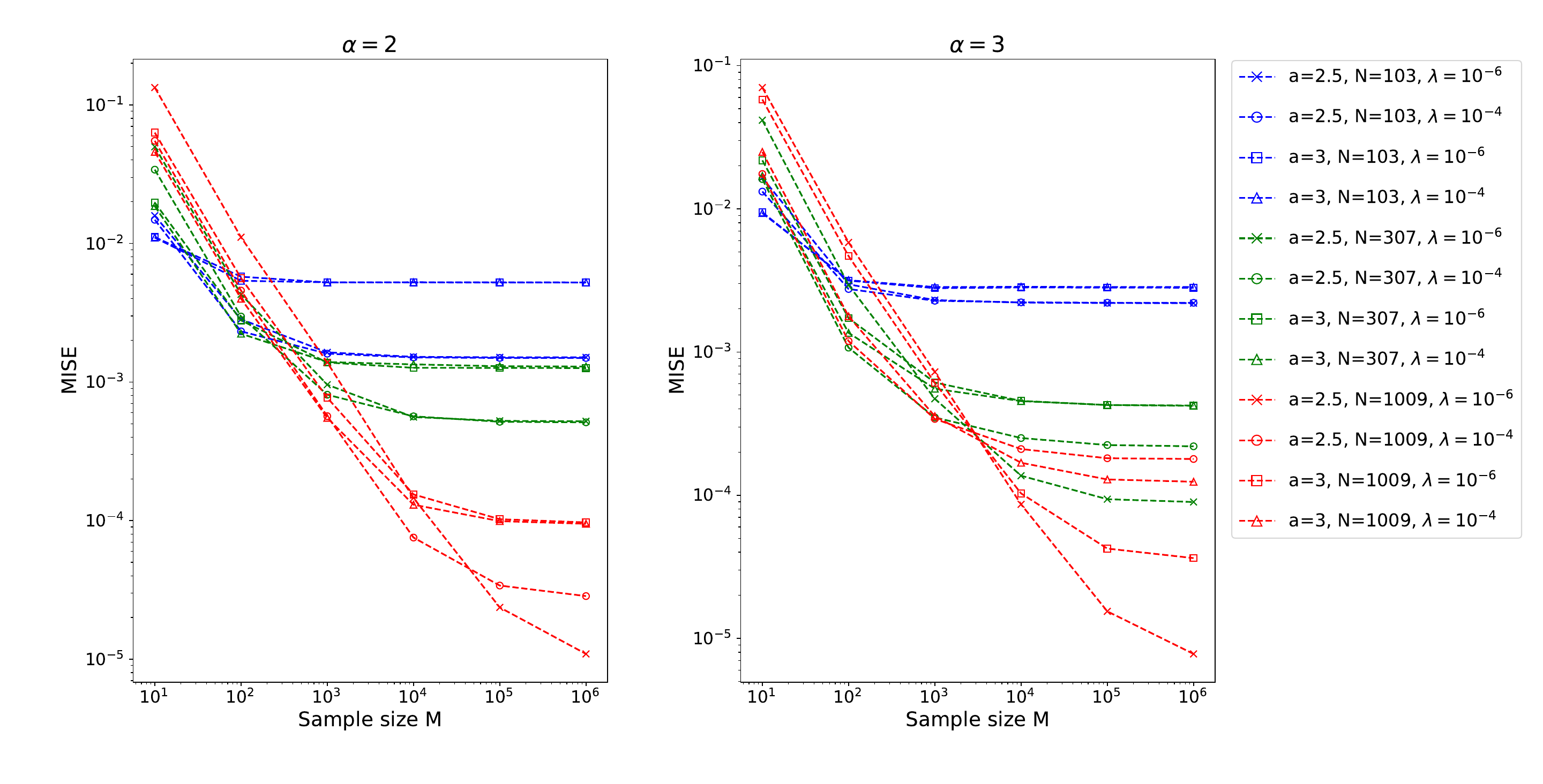}
  \caption{Plot of MISE convergence of the PSKK method for $N=103, 307, 1009,a=2.5,3,$ and $\lambda=10^{-4},10^{-6}$, with $\alpha=2$ (left), $3$ (right).  }
  \label{fig3:4dim_N} 
\end{figure}

Figure~\ref{fig3:4dim_N} shows the decay of MISE of the PSKK method for $N=103$, $ 307$, $1009$, with $a=2.5,3$ and $\lambda=10^{-4},10^{-6}$. In the left panel, we set $\alpha=2$, while in the right panel $\alpha=3$ is used. It is evident that, with a sufficiently large sample size $M$, the MISE of the PSKK method diminishes as $N$ increases. However, we report that increasing $N$ beyond $N=1009$ did not help decrease the error. We interpret this as the term involving $N$ in the MISE bound~(\ref{MISE_with_par}) being negligible when $N$ exceeds $1009$ under the current sample size.

\begin{figure}[htbp]
  \centering
  \includegraphics[width=\linewidth]{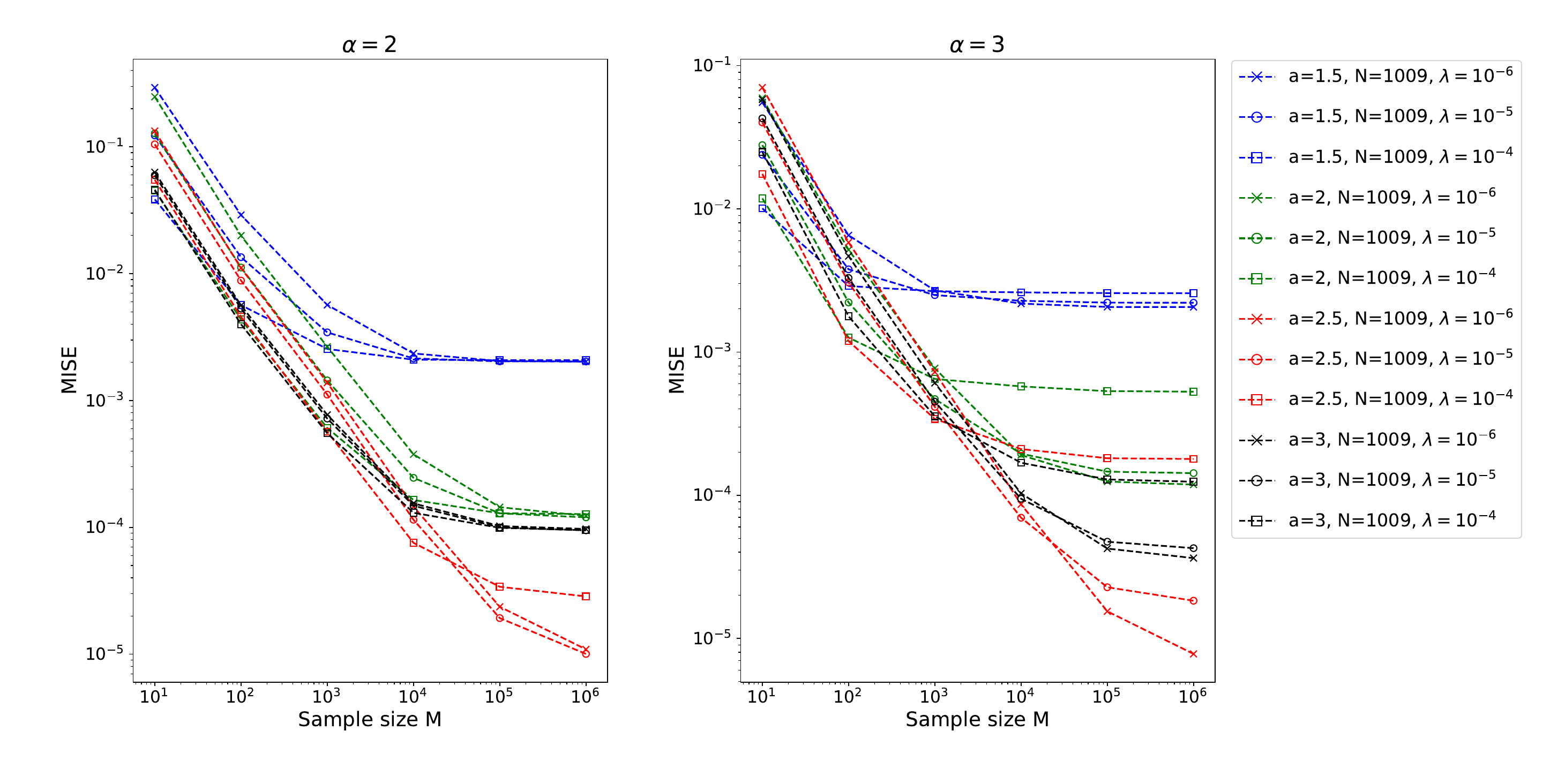}
  \caption{Plot of MISE convergence of the PSKK method for $N=1009,a=1.5,2,2.5,3,$ and $\lambda=10^{-4},10^{-5},10^{-6}$, with $\alpha=2$ (left), $3$ (right).}
  \label{fig4:4dim_a} 
\end{figure}

Next, we analyze the behavior of the MISE for varying values of $a$ while keeping $N$ fixed. Figure~\ref{fig4:4dim_a} illustrates the MISE decay of the PSKK method for $N = 1009$ and different parameter combinations: $a=1.5,2,2.5,3,\lambda=10^{-4},10^{-5},10^{-6}$ and $\alpha=2,3$. It can be observed that in the $4$-dimensional example, the MISE is highly sensitive to the parameter $a$. Excessively large or small values of $a$ can result in an increased MISE. Specifically, smaller values of $a$ lead to a significant difference between $f$ and $\widetilde{f}$, and the support of the estimator becomes relatively small, thereby increasing truncation errors. Conversely, larger values of $a$ increase the projection error and variance bound. For this particular example, when dealing with larger sample sizes, a recommended value for $a$ is approximately $2.5$.

\begin{figure}[htbp]
  \centering
  \includegraphics[width=\linewidth]{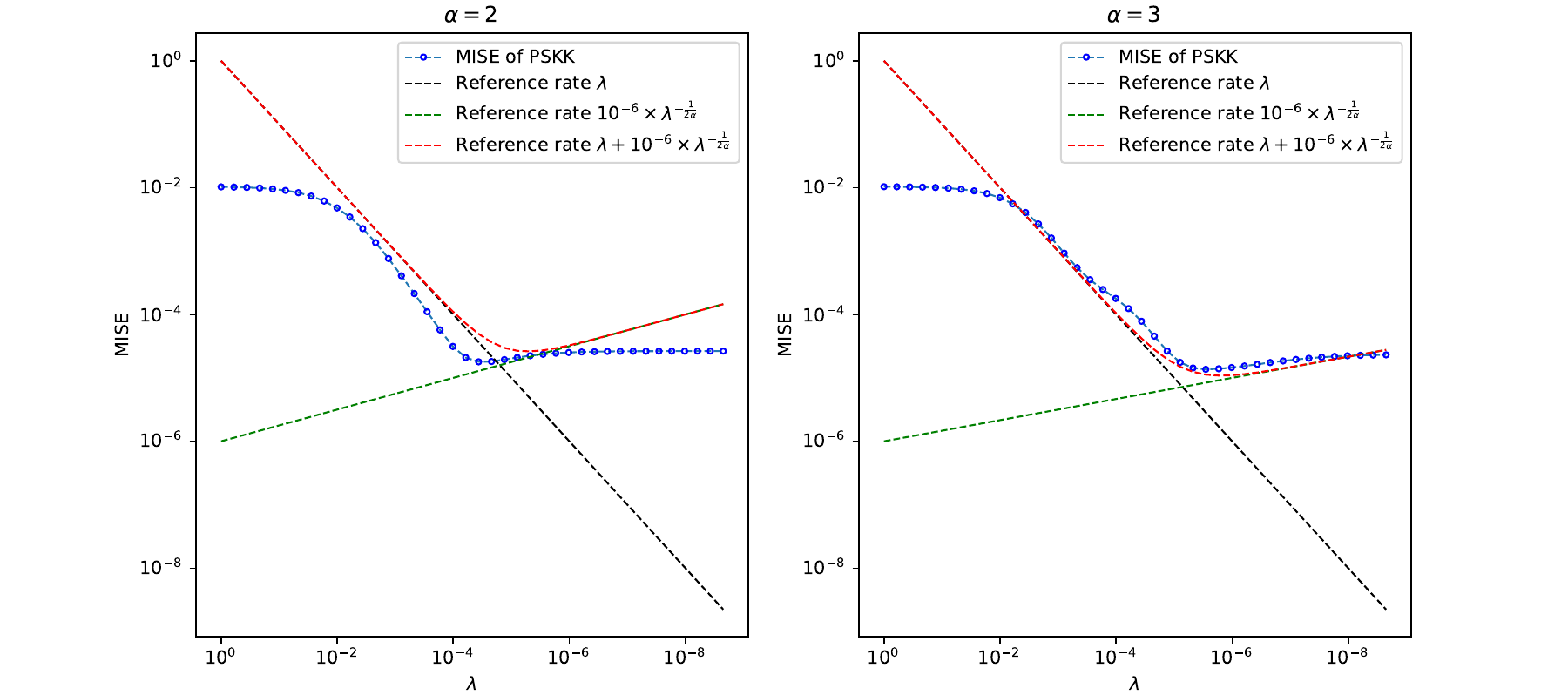}
  \caption{Plot of MISE for $\alpha=2$ (left), $3$ (right), $N=1009,a=2.5,$ and $M=10^5$ with $\lambda=0.6^k,k=0,1,\ldots,39$.}
  \label{fig5:4dim_lambda} 
\end{figure}
In Figure~\ref{fig5:4dim_lambda}, we fix $a=2.5,N=1009,M=10^5$ and analyze the relationship between MISE and $\lambda$. For $\alpha=2$ and $\alpha = 3$, we compute the MISE of the PSKK method for $\lambda=0.6^k$, where $k=0,1,\ldots,39$. In the case of $\alpha=3$ (see the right panel of Figure~\ref{fig5:4dim_lambda}), from $\lambda=0.6^8\approx0.017$ to $\lambda=0.6^{25}\approx 2.8\times 10^{-6}$, we see the MISE decays almost linearly in $\lambda$. Then, from $\lambda=0.6^{25}\approx 2.8\times 10^{-6}$ to $\lambda=0.6^{39}\approx 2.2\times 10^{-9}$, the MISE increases at a rate of $10^{-6}\lambda^{-\frac{1}{2\alpha}}=10^{-6}\lambda^{-\frac{1}{6}}$. The MISE is approximately controlled by $\lambda + 10^{-6}\lambda^{-\frac{1}{2\alpha}}$. For the case where $\alpha=2$ (see the left panel of Figure~\ref{fig5:4dim_lambda}), we observed similar results. These findings agree with the relationship between the MISE bound and $\lambda$ as outlined in (\ref{MISE_with_par}). Furthermore, we observe that for sufficiently small $\lambda$, the MISE increases slowly as $\lambda$ decreases.
Therefore, selecting a $\lambda$ smaller than the inflection point is acceptable in practice.

\section{Higher dimensional examples}\label{SM_higher_examples}
In this section, we consider higher dimensional Gaussian mixture distributions. We compare the MISE between the PSKK method and the KDE method using a density function similar to that in Subsection~\ref{Example2}, but with the component means $\bm{\mu}_k$ adapted to the $d$-dimensional space. Specifically, the modified means are:

\begin{equation}
    \bm{\mu}_{k}^{(4)}=\left(\left\{\frac{k}{9}\right\}-\frac{4}{9},\left\{\frac{2k}{9}\right\}-\frac{4}{9},\left\{\frac{4k}{9}\right\}-\frac{4}{9},\left\{\frac{8k}{9}\right\}-\frac{4}{9}\right)^\top,\nonumber
\end{equation}
\begin{equation}
    \bm{\mu}_{k}^{(5)}=\left(\left\{\frac{k}{9}\right\}-\frac{4}{9},\left\{\frac{2k}{9}\right\}-\frac{4}{9},\left\{\frac{4k}{9}\right\}-\frac{4}{9},\left\{\frac{6k}{9}\right\}-\frac{4}{9},\left\{\frac{8k}{9}\right\}-\frac{4}{9}\right)^\top,\nonumber
\end{equation}
\begin{equation}
    \bm{\mu}_{k}^{(6)}=\left(\left\{\frac{k}{9}\right\}-\frac{4}{9},\left\{\frac{2k}{9}\right\}-\frac{4}{9},\left\{\frac{4k}{9}\right\}-\frac{4}{9},\left\{\frac{6k}{9}\right\}-\frac{4}{9},\left\{\frac{7k}{9}\right\}-\frac{4}{9},\left\{\frac{8k}{9}\right\}-\frac{4}{9}\right)^\top.\nonumber
\end{equation}

\begin{table}[htbp]
  \centering
  \caption{MISE of PSKK and KDE}\label{tab:MISE of PSKK and KDE}
  \resizebox{0.9\textwidth}{!}{
  \begin{tabular}{ccccccc}
    \toprule
    \multirow{2}{*}{$M$} & \multicolumn{2}{c}{$d=4$} & \multicolumn{2}{c}{$d=5$} & \multicolumn{2}{c}{$d=6$} \\
    \cmidrule(lr){2-3} \cmidrule(lr){4-5} \cmidrule(lr){6-7}
                          & PSKK & KDE & PSKK & KDE & PSKK & KDE\\
    \midrule
    $10^1$ & $1.75\times 10^{-1}$ & $1.20\times 10^{-2}$ & $8.74\times 10^{-2}$ & $5.99\times 10^{-3}$ & $1.98\times 10^{-2}$ & $1.83\times 10^{-3}$ \\
    $10^2$ & $1.59\times 10^{-2}$ & $1.89\times 10^{-3}$ & $7.62\times 10^{-3}$ & $8.57\times 10^{-4}$ & $2.05\times 10^{-3}$ & $3.68\times 10^{-4}$ \\
    $10^3$ & $1.53\times 10^{-3}$ & $6.11\times 10^{-4}$ & $8.08\times 10^{-4}$ & $3.37\times 10^{-4}$ & $2.02\times 10^{-4}$ & $1.72\times 10^{-4}$ \\
    $10^4$ & $1.57\times 10^{-4}$ & $2.30\times 10^{-4}$ & $1.02\times 10^{-4}$ & $1.44\times 10^{-4}$ & $4.05\times 10^{-5}$ & $7.88\times 10^{-5}$ \\
    $10^5$ & $1.94\times 10^{-5}$ & $7.49\times 10^{-5}$ & $3.39\times 10^{-5}$ & $5.34\times 10^{-5}$ & $2.29\times 10^{-5}$ & $3.43\times 10^{-5}$ \\
    $10^6$ & $5.77\times 10^{-6}$ & $2.48\times 10^{-5}$ & $2.70\times 10^{-5}$ & $2.01\times 10^{-5}$ & $2.09\times 10^{-5}$ & $1.44\times 10^{-5}$ \\
    \bottomrule
  \end{tabular}
  }
\end{table}
Table~\ref{tab:MISE of PSKK and KDE} presents the MISE for both the PSKK method and the KDE method in dimensions $4,5$ and $6$. For the PSKK method, we set $N=1009,\lambda=10^{-6}$ and $\alpha=2$, with truncation $a=2.5$ for the $4$-dimensional case and $a=2$ for higher dimensions. Figures~\ref{fig3:4dim_N} and~\ref{fig4:4dim_a} demonstrate that the MISE of the PSKK method initially follows the theoretical convergence rate until reaching a ``turning point'', after which it stabilizes. This ``turning point'' depends on parameters including $N$, $a$ and $d$. As shown in Table~\ref{tab:MISE of PSKK and KDE}, ``turning point'' occurs earlier with increasing dimensionality. Furthermore, it should be mentioned that the convergence rate of the PSKK method will not change before reaching the ``turning point'' even for higher dimensions. Compared with the KDE method, the PSKK method has advantages in theoretical convergence rate, although achieving these rates requires more careful parameter selection than the simple fixed parameters used here.

\bibliographystyle{siamplain}
\bibliography{references}
\end{document}